\documentclass[11pt,dvipsnames,a4paper]{article}
\usepackage[affil-it]{authblk}
\usepackage[utf8]{inputenc}
\usepackage{amsmath, amssymb, mathrsfs, amsthm, amsfonts}
\usepackage{latexsym}
\usepackage{hyperref}
\usepackage{a4wide}
\usepackage{color,xcolor}
\usepackage{graphicx}
\usepackage{fancyhdr}
\usepackage{tikz}
\usepackage{multirow}
\usepackage{tikz-cd}
\usepackage{tikz-network}
\usetikzlibrary{arrows}

\newtheorem{theorem}{Theorem}[section]
\newtheorem{corollary}[theorem]{Corollary}
\newtheorem{lemma}[theorem]{Lemma}
\newtheorem{proposition}[theorem]{Proposition}

\theoremstyle{definition}
\newtheorem{definition}[theorem]{Definition}
\theoremstyle{remark}
\newtheorem{example}{Example}

\author{Hery Randriamaro
	\thanks{This research was funded by my mother \\
		Lot II B 32 bis Faravohitra, 101 Antananarivo, Madagascar \\
		e-mail: \texttt{hery.randriamaro@outlook.com}}}

\begin{document}

\thispagestyle{empty}

\noindent {\color{MidnightBlue} \rule{\linewidth}{2pt}}

\vspace*{15pt}

\noindent \textsl{\Huge A Survey of the Valuation Algebra\\ motivated by a Fundamental Application\\ to Dissection Theory}

\vspace*{20pt}

\noindent \textbf{\Large Hery Randriamaro${}^1$} 

\vspace*{10pt}

\noindent ${}^1$\: This research was funded by my mother \\ \textsc{Lot II B 32 bis Faravohitra, 101 Antananarivo, Madagascar} \\ \texttt{hery.randriamaro@outlook.com}

\vspace*{20pt}

\noindent \textsc{\large Abstract} \\ A lattice $L$ is said lowly finite if the set $[\mathsf{0},a]$ is finite for every element $a$ of $L$. We mainly aim to provide a complete proof that, if $M$ is a subset of a complete lowly finite distributive lattice $L$ containing its join-irreducible elements, and $a$ an element of $M$ which is not join-irreducible, then $\displaystyle \sum_{b \in M \cap [\mathsf{0},a]} \mu_M(b,a)b$ belongs to the submodule $\langle a \wedge b + a \vee b - a - b\ |\ a,b \in L \rangle$ of $\mathbb{Z}L$. That property was originally established by Zaslavsky for finite distributive lattice. It is essential to prove the fundamental theorem of dissection theory as will be seen. We finish with a concrete application of that theorem to face counting for submanifold arrangements.

\vspace*{10pt}

\noindent \textsc{Keywords}: Lattice, Valuation, Möbius Algebra, Subspace Arrangement

\vspace*{10pt}

\noindent \textsc{MSC Number}: 06A07, 06A11, 06B10, 57N80

\vspace*{10pt}

\noindent {\color{MidnightBlue} \rule{\linewidth}{2pt}}

\vspace*{10pt}

\section{Introduction}

\noindent A distributive lattice is a partially order set with join and meet operations which distribute over each other. The prototypical example of a such structure are the sets where join and meet are the usual union and intersection. Further examples include the Lindenbaum algebra of logics that support conjunction and disjunction, every Heyting algebra, and Young's lattice whose join and meet of two partitions are given by the union and intersection of the corresponding Young diagrams. This survey mainly aims to provide a complete proof that, if $L$ is a complete lowly finite distributive lattice, $M$ a subset of $L$ containing its join-irreducible elements, $f:L \rightarrow G$ a valuation on $L$ to a module $G$, and $a$ an element of $M$ which is not join-irreducible, then
\begin{equation} \label{EqM}
\sum_{b \in [\mathsf{0},a] \cap M} \mu_M(b,a)f(b) = 0.
\end{equation}

\noindent The proof is carried out in several stages. We first have to consider the general case of posets in Section~\ref{SecPo}. A proof of Zorn's lemma and an introduction to the Möbius algebra $\mathrm{M\ddot{o}b}(L)$ of a lowly finite poset $L$ are namely provided. The lemma bearing his name was proved by Zorn in 1935 \cite{Zo}. Although we can easily find diverse proofs of that lemma in the literature, new ones are still proposed other time like that of Lewin in 1991 \cite{Le}. Ours is inspired by the lecture notes of Debussche \cite[§~2.II]{De}. The Möbius algebra was discovered in 1967 by Solomon who defined it for finite posets \cite[§~2]{So}. We establish the Möbius inversion formula, and prove that $\displaystyle \bigg\{\sum_{b \in [\mathsf{0},a]} \mu_L(b,a)b\ \bigg|\ a \in L\bigg\}$ is a complete set of orthogonal idempotents in $\mathrm{M\ddot{o}b}(L)$.

\smallskip

\noindent We study the special case of lattices in Section~\ref{SecLa}. After viewing some essential generalities, we focus on the distributive lattices, and establish diverse properties like the distributivity of a lattice $L$ if and only if, for all $a,b,c \in L$, $c \vee a = c \vee b$ and $c \wedge a = c \wedge b$ imply $a = b$.

\smallskip

\noindent Those properties are necessary to investigate the valuation algebra in Section~\ref{SecVa}. It is the central part of this survey, and principally inspired from the articles of Geissinger \cite{Ge1}, \cite[§~3]{Ge2} and that of Zaslavsky \cite[§~2]{Za}. In that section is particularly proved that if $M$ is a subset of a complete lowly finite distributive lattice $L$ containing its join-irreducible elements, and $a$ an element of $M$ which is not join-irreducible, then $\displaystyle \sum_{b \in M \cap [\mathsf{0},a]} \mu_M(b,a)b$ belongs to the submodule $\langle a \wedge b + a \vee b - a - b\ |\ a,b \in L \rangle$ of $\mathbb{Z}L$. It is the property that allows to deduce Equation~\ref{EqM}.

\smallskip

\noindent Thereafter, we use Equation~\ref{EqM} to deduce the fundamental theorem of dissection theory in Section~\ref{SecDi}, that is, if $\mathscr{A}$ is a subspace arrangement in a topological space $T$ with $|\chi(T)| < \infty$, and $L$ a meet-refinement of $L_{\mathscr{A}}$, then $\displaystyle \sum_{C \in C_{\mathscr{A}}}\chi(C) = \sum_{X \in L} \mu_L(X,T) \chi(X)$. In its original form of 1977 \cite[Theorem~1.2]{Za}, Zaslavsky expressed it for CW complexes. The number of chambers is consequently $\displaystyle \#C_{\mathscr{A}} = \frac{1}{c} \sum_{X \in L} \mu_L(X,T) \chi(X)$ if $c$ is the Euler characteristic of every chamber. Deshpande showed a similar result in 2014 for the special case of a submanifold arrangement with chambers having the same Euler characteristic $(-1)^l$ \cite[Theorem~4.6]{Des}.

\smallskip

\noindent We finally compute the $\mathrm{f}$-polynomial of submanifold arrangements from the dissection theorem of Zaslavsky in Section~\ref{SecCo}. Face counting of topological space has doubtless its origins in the formulas for the numbers of the $i$-dimensional faces if planes in $\mathbb{R}^3$ fall into $k$ parallel families in general position established by Steiner in 1826 \cite{Ste}. About 150 years later, Alexanderson and Wetzel computed the same numbers but for an arbitrary set of planes \cite{AlWe}, and Zaslavsky for hyperplane arrangements in a Euclidean space of any dimension \cite[Theorem~A]{Z}. One of our formulas is a generalization of those results as it considers a submanifold arrangement $\mathscr{A}$ such that $\chi(X) = (-1)^{\dim X}$ for every $X \in L_{\mathscr{A}} \cup F_{\mathscr{A}}$, and states that $\mathrm{f}_{\mathscr{A}}(x) = (-1)^{\mathrm{rk}\,\mathscr{A}} \mathrm{M}_{\mathscr{A}}(-x,-1)$ where $\mathrm{M}_{\mathscr{A}}$ is the Möbius polynomial of $\mathscr{A}$. Moreover, Pakula computed the number of chambers of a pseudosphere arrangement with simple complements in 2003 \cite[Corollary~1]{Pa}. Another formula is a generalization of his result considering a submanifold arrangement $\mathscr{A}$ such that
$$\forall C \in F_{\mathscr{A}}:\, \chi(C) = (-1)^{\dim C} \quad \text{and} \quad \forall X \in L_{\mathscr{A}}:\, \chi(X) = \begin{cases}
	2 & \text{if}\ \dim X \equiv 0 \mod 2 \\
	0 & \text{otherwise}	
\end{cases},$$
and states
$$\mathrm{f}_{\mathscr{A}}(x) = (-1)^{n-\mathrm{rk}\,\mathscr{A}} \big(\mathrm{M}_{\mathscr{A}}(x,-1) + \gamma_n \mathrm{M}_{\mathscr{A}}(-x,-1)\big)\ \text{with}\  \gamma_n := \begin{cases}
	1 & \text{if}\ \dim X \equiv 0 \mod 2 \\
	-1 & \text{otherwise}	
\end{cases}.$$

\section{Poset} \label{SecPo}

\noindent We begin with the general case of posets. The Zorn's lemma is especially proved, and the Möbius algebra described as it plays a key role in this survey.

\begin{definition}
A \textbf{partial order} is a binary relation $\preceq$ over a set $L$ such that, for $a,b,c \in L$,
\begin{itemize}
\item $a \preceq a$,
\item if $a \preceq b$ and $a \succeq b$, then $a=b$,
\item if $a \preceq b$ and $b \preceq c$, then $a \preceq c$.
\end{itemize}
The set $L$ with a partial order is called a \textbf{partially ordered set} or \textbf{poset}, and two elements $a,b \in L$ are said comparable if $a \preceq b$ or $a \succeq b$. 
\end{definition}

\begin{definition}
	A poset $L$ has an uppest resp. lowest element $\mathsf{1}$ resp. $\mathsf{0} \in L$ if, for every $a \in L$, one has $a \preceq \mathsf{1}$ resp. $a \succeq \mathsf{0}$. The poset is said \textbf{complete} if it has an uppest and a lowest element.
\end{definition}

\subsection{Zorn's Lemma}

\begin{definition}
A subset $C$ of a poset $P$ is a \textbf{chain} if any two elements in $C$ are comparable. 
\end{definition}

\noindent Denote by $\mathcal{C}_L$ the set formed by the chains of a poset $L$. A subset $S$ of $L$ has an upper resp. lower bound if there exists $u$ resp. $l \in L$ such that $s \preceq u$ resp. $l \preceq s$ for each $s \in S$. The upper resp. lower bound $u$ resp. $l$ is said strict if $u$ resp. $l \notin S$.

\begin{definition}
A poset $L$ is said \textbf{inductive} if every chain included in $L$ has an upper bound.
\end{definition}

\noindent For an inductive poset $L$, and $C \in \mathcal{C}_L$, let $C_{\prec}$ be the set formed by the strict upper bound of $C$, and denote by $\mathcal{E}_L$ the set $\{C \in \mathcal{C}_L\ |\ C_{\prec} \neq \emptyset\}$. The axiom of choice allows to deduce the existence of a function $\mathrm{c}:2^L \setminus \{\emptyset\} \rightarrow L$ such that, for every $A \in 2^L \setminus \{\emptyset\}$, we have $\mathrm{c}(A) \in A$. Define the function $\mathrm{m}: \mathcal{E}_L \rightarrow L$, for $C \in \mathcal{E}_L$, by $\mathrm{m}(C) := \mathrm{c}(C_{\prec})$.

\begin{definition}
Let $S,A$ be subsets of a poset $L$. The set $S$ is called a \textbf{segment} of $A$ if
$$S \subseteq A \quad \text{and} \quad \forall s \in S,\, \forall a \in A:\ s \succeq a \Rightarrow a \in S.$$
\end{definition}

\begin{definition}
An upper resp. lower bound $u$ resp. $l$ of the subset $S$ of a poset $L$ is called a \textbf{join} resp. \textbf{meet} if $u \preceq a$ resp. $b \preceq l$ for each upper resp. lower bound $a$ resp. $b$ of $S$.	
\end{definition}

\begin{definition}
A chain $C$ of an inductive poset $L$ is called a \textbf{good set} if, for every segment $S$ of $C$ with $S \neq C$, we have $S_{\prec} \cap C \neq \emptyset$ and $\mathrm{m}(S)$ is the meet of $S_{\prec} \cap C$.
\end{definition}

\noindent For elements $a,b$ of a poset, by $a \prec b$ we mean that $a \preceq b$ and $a \neq b$.

\begin{lemma} \label{LeGo}
Let $A,B$ be nonempty good sets of an inductive poset $L$. Then, either $A$ is a segment of $B$ or vice versa.	
\end{lemma}

\begin{proof}
Note first that $\emptyset$ is a chain of $L$. As $L$ is inductive, $\emptyset$ has then an upper bound in $L$ which is necessary a strict upper bound, hence $\emptyset \in \mathcal{E}_L$. Moreover, since $\emptyset$ is obviously a segment of both $A$ and $B$ which are good sets, then $\mathrm{m}(\emptyset) \in \emptyset_{\prec} \cap A \cap B$ and $A \cap B \neq \emptyset$.

\noindent For $a \in A \cap B$, the sets $S_{a,A} := \{s \in A\ |\ s \prec a\}$ and $S_{a,B} := \{s \in B\ |\ s \prec a\}$ are clearly segments of $A$ and $B$ respectively. Set $C := \{a \in A \cap B\ |\ S_{a,A} = S_{a,B}\}$, and let $b \in C$, $c \in A$, with $b \succ c$. We have $c \in S_{b,A} = S_{b,B}$, then $c \in B$ which implies $c \in A \cap B$. If $d \in S_{c,A}$, then $d \prec c \prec b$ implies $d \in S_{b,A} = S_{b,B}$, hence $b \in S_{c,B}$ and $S_{c,A} \subseteq S_{c,B}$. Similarly, we have $S_{c,B} \subseteq S_{c,A}$, then $c \in C$. Therefore, $C$ is a segment of $A$ and $B$.

\noindent Suppose now that $C \neq A$ and $C \neq B$. As $A,B$ are good sets, then $\mathrm{m}(C) \in A \cap B$. Remark that $C \sqcup \big\{\mathrm{m}(C)\big\} = S_{\mathrm{m}(C),A} = S_{\mathrm{m}(C),B}$, then $\mathrm{m}(C) \in C$ which is absurd. Hence $C = A$ or $C = B$, in other words, $A$ is a segment of $B$ or vice versa.
\end{proof}

\noindent Denoting by $\mathcal{G}_L$ the set formed by the good sets of an inductive poset $L$, set $\displaystyle U_L := \bigcup_{A \in \mathcal{G}_L}A$.

\begin{theorem}
If $L$ is an inductive poset, then $U_L$ is a good set.
\end{theorem}

\begin{proof}
For $a,b \in U_L$, there exist good sets $S_a, S_b$ such that $a \in S_a$ and $b \in S_b$. Using Lemma~\ref{LeGo}, we get either $S_a \subseteq S_b$ or $S_b \subseteq S_a$. That means either $a \preceq b$ or $a \succeq b$, and $U_L$ is consequently a chain.

\noindent Let $A \in \mathcal{G}_L$, $a \in A$, and $b \in U_L$ with $a \succeq b$. There is $B \in \mathcal{G}_L$ with $b \in B$. From Lemma~\ref{LeGo},
\begin{itemize}
\item if $A$ is a segment of $B$, then $A$ is a segment and $b \in A$,
\item if $B$ is a segment of $A$, then $B \subseteq A$ and $b \in A$.
\end{itemize}
In any case, we have $b \in A$, then $A$ is a segment of $U_L$.

\noindent Consider a segment $S$ of $U_L$ such that $S \neq U_L$. Since $U_L$ is a chain, necessarily $U_L \setminus S \subseteq S_{\prec}$. Let $a \in U_L \setminus S$, and $A \in \mathcal{G}_L$ such that $a \in A$. As $A$ is a segment of $U_L$, then $S \varsubsetneq A$ and $S$ is a segment of $A$. Moreover, $\mathrm{m}(S)$ is the meet of $S_{\prec} \cap A$. If there exists $b \in S_{\prec} \cap U_L$ such that $b \prec \mathrm{m}(S)$, we would get $b \in A$, which is absurd. Therefore, $\mathrm{m}(S)$ is the meet of $S_{\prec} \cap U_L$, and $U_L$ is a good set.
\end{proof}

\begin{definition}
An element $a$ of a poset $L$ is said \textbf{maximal} if there does not exist an element $b \in L \setminus \{a\}$ such that $b \succ a$. 
\end{definition}

\begin{corollary}[Zorn's Lemma] \label{CoZo}
Every inductive poset $L$ has a maximal element.	
\end{corollary}

\begin{proof}
Recall that, since $U_L$ is a chain, it consequently possesses an upper bound. Suppose $U_{L \prec} \neq \emptyset$, and let $u \in U_{L \prec}$. Then $U_L \sqcup \{u\}$ is a good set which is absurd. Hence, $U_L$ has a unique upper bound, contained in $U_L$, which is a maximal element of $L$.
\end{proof}

\subsection{Möbius Algebra}

\noindent For two elements $a,b$ of a poset $L$ such that $a \preceq b$, denote by $[a,b]$ the set $\{c \in L\ |\ a \preceq c \preceq b\}$.

\begin{definition}
A poset $L$ is \textbf{locally finite} if, for all $a,b \in L$ such that $a \preceq b$, $[a,b]$ is finite.
\end{definition}

\noindent For a locally finite poset $L$, denote by $\mathsf{Inc}(L)$ the module of the functions $f:L^2 \rightarrow \mathbb{Z}$ with the property that, if $x,y \in L$, then $f(x,y) = 0$ if $x \npreceq y$.

\begin{definition}
The \textbf{incidence algebra} $\mathrm{Inc}(L)$ of a locally finite poset $L$ is the module of functions $f:L^2 \rightarrow \mathbb{Z}$, having the property $f(a,b) = 0$ if $a \npreceq b$, with distributive multiplication $h = f \cdot g$ defined, for $f,g \in \mathrm{Inc}(L)$, by
$$h(a,b):=0\ \text{if}\ a \npreceq b \quad \text{and} \quad h(a,b) := \sum_{c \in [a,b]}f(a,c)g(c,b)\ \text{otherwise}.$$
Its multiplicative identity is the Kronecker delta $\delta:L^2 \rightarrow \mathbb{Z}$ with $\delta(a,b) := \begin{cases}
1 & \text{if}\ a=b, \\
0 & \text{otherwise}
\end{cases}$.
\end{definition}

\begin{definition}
For a locally finite poset $L$, the \textbf{zeta function} $\zeta_L$ and the \textbf{Möbius function} $\mu_L$ in the incidence algebra $\mathrm{Inc}(L)$ are defined, for $a,b \in L$ with $a \preceq b$, by
$$\zeta_L(a,b):= 1 \quad \text{and} \quad \mu_L(a,b) := \begin{cases}
1 & \text{if}\ a=b \\
\displaystyle -\sum_{\substack{c \in [a,b] \\ c \neq b}} \mu_L(a,c) = -\sum_{\substack{c \in [a,b] \\ c \neq a}} \mu_L(c,b) & \text{otherwise} 
\end{cases}.$$
\end{definition}

\begin{lemma} \label{LeZe}
For a locally finite poset $L$, the zeta function is the multiplicative inverse of the Möbius function in the incidence algebra $\mathrm{Inc}(L)$.
\end{lemma}

\begin{proof}
For $a,b \in L$ with $a \preceq b$, we have $\zeta_L \cdot \mu_L(a,a) = \mu_L \cdot \zeta_L(a,a) = 1 = \delta(a,a)$, but also
$$\zeta_L \cdot \mu_L(a,b) = \sum_{c \in [a,b]} \mu_L(c,b) = 0 = \delta(a,b) \quad \text{and} \quad \mu_L \cdot \zeta_L(a,b) = \sum_{c \in [a,b]} \mu_L(a,c) = 0 = \delta(a,b).$$
\end{proof}

\noindent The proof of this proposition is inspired from the original proof of Rota \cite[Proposition~2]{Ro}.

\begin{proposition}[Möbius Inversion Formula]
Let $L$ be a locally finite poset, $a,b \in L$ with $a \preceq b$, and $f,g$ two functions from $L$ onto a module $M$ over $\mathbb{Z}$. Then,
$$\forall x \in [a,b]:\,g(x) = \sum_{c \in [a,x]} f(c) \quad \Longleftrightarrow \quad \forall x \in [a,b]:\,f(x) = \sum_{c \in [a,x]} g(c) \mu_L(c,x).$$
\end{proposition}

\begin{proof}
Assume first that, for every $x \in [a,b]$, $\displaystyle g(x) = \sum_{c \in [a,x]} f(c)$. Using Lemma~\ref{LeZe}, we get
\begin{align*}
\sum_{c \in [a,x]} g(c) \mu_L(c,x) & = \sum_{c \in [a,x]} \sum_{d \in [a,c]} f(d) \mu_L(c,x) = \sum_{c \in [a,x]} \sum_{d \in [a,c]} f(d) \zeta_L(d,c) \mu_L(c,x) \\
& = \sum_{d \in [a,c]} \sum_{c \in [a,x]} f(d) \zeta_L(d,c) \mu_L(c,x) = \sum_{d \in [a,c]} f(d) \sum_{c \in [a,x]} \zeta_L(d,c) \mu_L(c,x) \\
& = \sum_{d \in [a,c]} f(d)\, \zeta_L \cdot \mu_L(d,x) = \sum_{d \in [a,c]} f(d) \delta(d,x) \\
& = f(x).
\end{align*}
Similarly, if $\displaystyle f(x) = \sum_{c \in [a,x]} g(c) \mu_L(c,x)$ for every $x \in [a,b]$, we obtain
\begin{align*}
\sum_{c \in [a,x]} f(c) & = \sum_{c \in [a,x]} \sum_{d \in [a,c]} g(d) \mu_L(d,c) = \sum_{c \in [a,x]} \sum_{d \in [a,c]} g(d) \mu_L(d,c) \zeta_L(c,x) = \sum_{d \in [a,c]} g(d) \delta(d,x) \\
& = g(x).
\end{align*}
\end{proof}

\begin{definition}
We say that a poset $L$ is \textbf{lowly finite} if the set $\{b \in L\ |\ b \preceq a\}$ is finite for any $a \in L$.
\end{definition}

\noindent For a lowly finite poset $L$ and $a \in L$, let $u_L(a)$ be the element $\displaystyle \sum_{\substack{c \in L \\ c \preceq a}} \mu_L(c,a)c$ of the module $\mathbb{Z}L$.

\begin{definition}
The \textbf{Möbius Algebra} $\mathrm{M\ddot{o}b}(L)$ of a lowly finite poset $L$ is the module $\mathbb{Z}L$ with distributive multiplication defined, for $a,b \in L$, by
$$a \cdot b := \sum_{\substack{c \in L \\ c \preceq a,\, c \preceq b}} u_L(c).$$ 
\end{definition}

\noindent Remark that the Möbius algebra was initial defined for finite posets \cite[§~2]{So}. For a lowly finite poset $L$ with a lowest element, define the algebra $\mathrm{A}_L := \langle \alpha_a\ |\ a \in L\rangle$ over $\mathbb{Z}$ with multiplication
$$\alpha_a \alpha_b := \begin{cases}
	\alpha_a & \text{if}\ a=b,\\
	0 & \text{otherwise}
\end{cases}.$$
To each $a \in L$, associate an element $a' \in \mathrm{A}_L$ by setting $\displaystyle a' := \sum_{b \in [\mathsf{0},a]} \alpha_b$.

\begin{lemma} \label{LeBas}
For a lowly finite poset $L$ with a lowest element, the set $\{a'\ |\ a \in L\}$ forms a basis of the algebra $\mathrm{A}_L$. 
\end{lemma}

\begin{proof}
From the Möbius inversion formula, we get $\displaystyle \alpha_a = \sum_{b \in [\mathsf{0},a]} \mu(b,a) b'$. The set $\{a'\ |\ a \in L\}$ consequently generates $\mathrm{A}_L$. Suppose that there exists a finite set $I \subseteq L$ and an integer set $\{i_a\}_{a \in I}$ such that $\displaystyle \sum_{a \in I} i_a a' = 0$. If $b$ is a maximal element of $I$, then $\displaystyle \alpha_b \sum_{a \in I} i_a a' = i_b \alpha_b = 0$, hence $i_b = 0$. Inductively, we deduce that $i_a = 0$ for every $a \in I$. The set $\{a'\ |\ a \in L\}$ is therefore independent.
\end{proof}

\noindent The following results were initially established by Greene for finite lattice \cite[§~1]{Gre}.

\begin{theorem}
For a lowly finite poset $L$ with a lowest element, the map $\phi: L \rightarrow \mathrm{A}_L,\, a \mapsto a'$ extends to an algebra isomorphism from $\mathrm{M\ddot{o}b}(L)$ to $\mathrm{A}_L$.
\end{theorem}

\begin{proof}
	The map $\phi$ clearly becomes a module homomorphism by linear extension, and an isomorphism by Lemma~\ref{LeBas}. Moreover, for $a,b \in L$,
	\begin{align*}
		\phi(a \cdot b) & = \phi\Big(\sum_{c \in [\mathsf{0},a] \cap [\mathsf{0},b]} u_L(c)\Big) = \sum_{c \in [\mathsf{0},a] \cap [\mathsf{0},b]} \phi\Big(\sum_{d \in [\mathsf{0},c]} \mu(d,c)d\Big) \\
		& = \sum_{c \in [\mathsf{0},a] \cap [\mathsf{0},b]} \sum_{d \in [\mathsf{0},c]} \mu(d,c)d' = \sum_{c \in [\mathsf{0},a] \cap [\mathsf{0},b]} \alpha_c,
	\end{align*}
	and $\displaystyle \phi(a) \phi(b) = a'b' = \sum_{c \in [\mathsf{0},a]} \alpha_c \times \sum_{d \in [\mathsf{0},b]} \alpha_d = \sum_{c \in [\mathsf{0},a] \cap [\mathsf{0},b]} \alpha_c$. Then $\phi(a \cdot b) = \phi(a) \phi(b)$, and $\phi$ is consequently an algebra isomorphism.
\end{proof}

\begin{corollary} \label{CoBaI}
For a lowly finite poset $L$ with a lowest element, the set $\{u_L(a)\ |\ a \in L\}$ is a complete set of orthogonal idempotents in $\mathrm{M\ddot{o}b}(L)$.
\end{corollary}

\begin{proof}
	Since $\displaystyle \phi\big(u_L(a)\big) = \sum_{b \in [\mathsf{0},a]} \mu_L(b,a)b' = \alpha_a$, then $\{u_L(a)\ |\ a \in L\}$ is a basis of $\mathrm{M\ddot{o}b}(L)$. Moreover $\phi\big(u_L(a) \cdot u_L(a)\big) = \alpha_a = \phi\big(u_L(a)\big)$, so the $u_L(a)$'s are idempotents.\\ Finally $\phi\big(u_L(a) \cdot u_L(b)\big) = \alpha_a \alpha_b = 0$ if $a \neq b$, hence the $u_L(a)$'s are orthogonal.
\end{proof}

\begin{corollary} \label{CoHoA}
Let $L$ be a lowly finite poset with a lowest element, and $M$ a subset of $L$ containing $\mathsf{0}$. Then, the linear map $\mathrm{j}: \mathrm{M\ddot{o}b}(L) \rightarrow \mathrm{M\ddot{o}b}(M)$, which on the basis $\{u_L(a)\ |\ a \in L\}$ has the values
$$\mathrm{j}\big(u_L(a)\big) := \begin{cases}
u_M(a) & \text{if}\ a \in M,\\
0 & \text{otherwise},
\end{cases}$$
is an algebra homomorphism.
\end{corollary}

\begin{proof}
Using Corollary~\ref{CoBaI}, $\mathrm{j}\big(u_L(a) \cdot u_L(a)\big) = \mathrm{j}\big(u_L(a)\big) = u_M(a) = \mathrm{j}\big(u_L(a)\big) \cdot \mathrm{j}\big(u_L(a)\big)$ if $a \in M$. Otherwise, $\mathrm{j}\big(u_L(a) \cdot u_L(a)\big) = 0 = \mathrm{j}\big(u_L(a)\big) \cdot \mathrm{j}\big(u_L(a)\big)$. For $a,b \in L$ with $a \neq b$, $\mathrm{j}\big(u_L(a) \cdot u_L(b)\big) = 0 = \mathrm{j}\big(u_L(a)\big) \cdot \mathrm{j}\big(u_L(b)\big)$.
\end{proof}

\section{Lattice} \label{SecLa}

\noindent We study the special but important case of lattices. After viewing some generalities, we focus on distributive ones, and establish diverse properties which are necessary to investigate the valuation algebra in the next section.

\begin{definition}
	A poset $L$ is a join-semilattice resp. meet-semilattice if each $2$-element subset $\{a,b\} \subseteq L$ has a join resp. meet denoted by $a \vee b$ resp. $a \wedge b$. It is called a \textbf{lattice} if $L$ is both a join- and meet-semilattice, moreover $\vee$ and $\wedge$ become binary operations on $L$. 
\end{definition}

\begin{proposition} \label{LeLow}
If a lattice $L$ is lowly finite, then it has a lowest element.
\end{proposition}

\begin{proof}
	For any $a \in L$, the principal ideal $\mathrm{id}(a)$ has a lowest element which is $\displaystyle \mathsf{0}_a := \bigwedge_{x \in \mathrm{id}(a)}x$. Consider $b \in L \setminus \{a\}$ and the lowest element $\mathsf{0}_b$ of $\mathrm{id}(b)$. The fact $\mathsf{0}_a \wedge \mathsf{0}_b \neq \mathsf{0}_a$ would contradict the fact that $\mathsf{0}_a$ is the lowest element of $\mathrm{id}(a)$. Hence, $L$ has a lowest element $\mathsf{0}$.
\end{proof}

\subsection{Generalities on Lattice}

\begin{definition}
	A \textbf{sublattice} of a lattice $L$ is a nonempty subset $M \subseteq L$ such that, for all $a,b \in M$, we have $a \vee b \in M$ and $a \wedge b \in M$.
\end{definition}

\begin{definition}
	A \textbf{lattice homomorphism} is a function $\varphi: L_1 \rightarrow L_2$ between two lattices $L_1$ and $L_2$ such that, for all $a,b \in L_1$,
	$$\varphi(a \vee b) = \varphi(a) \vee \varphi(b) \quad \text{and} \quad \varphi(a \wedge b) = \varphi(a) \wedge \varphi(b).$$
\end{definition}

\begin{definition}
	An \textbf{ideal} of a lattice $L$ is a sublattice $I \subseteq L$ such that, for any $a \in I$ and $b \in L$, we have $a \wedge b \in I$. If in addition $I \neq L$ and, for any $a \wedge b \in I$, either $a \in I$ or $b \in I$, then $I$ is a \textbf{prime ideal}.
\end{definition}

\begin{definition}
	Dually, a \textbf{filter} of a lattice $L$ is a sublattice $F \subseteq L$ such that, for any $a \in F$ and $b \in L$, we have $a \vee b \in F$. If in addition $F \neq L$ and, for any $a \vee b \in F$, either $a \in F$ or $b \in F$, then $F$ is a \textbf{prime filter}.
\end{definition}

\begin{proposition} \label{LeDu}
	A subset $M$ of a lattice $L$ is a prime ideal if and only if the subset $L \setminus M$ is a prime filter.
\end{proposition}

\begin{proof}
	Assume that $M$ is a prime ideal:
	\begin{itemize}
		\item If $a,b \in L \setminus M$, clearly $a \wedge b \in L \setminus M$ and $a \vee b \in L \setminus M$ since $(a \vee b) \wedge b \in L \setminus M$, then $L \setminus M$ is a sublattice.
		\item If $a \in M$ and $b \in L \setminus M$, once again $a \vee b \in L \setminus M$ since $(a \vee b) \wedge b \in L \setminus M$, then $L \setminus M$ is a filter.
		\item If $a \vee b \in L \setminus M$, it is clear that both $a,b$ cannot be all in $M$, then $L \setminus M$ is prime.
	\end{itemize}
	One similarly proves that if $M$ is a prime filter, then $L \setminus M$ is a prime ideal.
\end{proof}

\begin{definition}
Let $L$ be a lattice, and $a \in L$. The \textbf{principal ideal} generated by $a$ is the ideal $\mathrm{id}(a) := \{b \in L\ |\ b \preceq a\}$, dually the \textbf{principal filter} generated by $a$ is the filter $\mathrm{fil}(a) := \{b \in L\ |\ b \succeq a\}$.
\end{definition}

\begin{definition}
	An element $a$ of a lattice $L$ is \textbf{join-irreducible} if, for any subset $S \subseteq L$, $\displaystyle a = \bigvee_{b \in S}b$ implies $a \in S$. Denote by $\mathrm{ji}(L)$ the set formed by the join-irreducible elements of $L$.
\end{definition}

\begin{lemma} \label{LeJi}
	Let $L$ be a lattice, and $a \in L$. Then, $a \in \mathrm{ji}(L)$ if and only if $\displaystyle a \neq \bigvee_{\substack{b \in L \\ b \prec a}} b$.
\end{lemma}

\begin{proof}
	If $a \in \mathrm{ji}(L)$, as $a \notin \{b \in L\ |\ b \prec a\}$, then $\displaystyle a \neq \bigvee_{\substack{b \in L \\ b \prec a}} b$.
	
	\noindent Assume now that $\displaystyle a \neq \bigvee_{\substack{b \in L \\ b \prec a}} b$, and let $S \subseteq L$ such that $\displaystyle a = \bigvee_{b \in S}b$. Since $b \preceq a$ for every $b \in S$, the only possibility is $a \in S$, and consequently $a \in \mathrm{ji}(L)$.
\end{proof}

\noindent The proof of the following proposition takes inspiration from that of Bhatta and Ramananda \cite[Proposition~2.2]{BhRa}.

\begin{proposition} \label{PrJi}
	Let $L$ be a lowly finite lattice, and $a \in L$. Then, $\displaystyle a = \bigvee_{b \, \in \, \mathrm{id}(a) \cap \mathrm{ji}(L)}b$.
\end{proposition}

\begin{proof}
	It is obvious if $a \in \mathrm{ji}(L)$. Now, assume that $a \in L \setminus \mathrm{ji}(L)$ and $\displaystyle a \neq \bigvee_{b \, \in \, \mathrm{id}(a) \cap \mathrm{ji}(L)}b$. The set $\displaystyle S = \Big\{x \in L\ \Big|\ x \neq \bigvee_{b \, \in \, \mathrm{id}(x) \cap \mathrm{ji}(L)}b\Big\}$ is nonempty and has a minimal element $c$ as $L$ is lowly finite. Since $\displaystyle c \neq \bigvee_{b \, \in \, \mathrm{id}(c) \cap \mathrm{ji}(L)}b$, then $c \notin \mathrm{ji}(L)$, and it follows from Lemma~\ref{LeJi} that $\displaystyle c = \bigvee_{\substack{b \in L \\ b \prec c}} b$. Clearly, $c$ is an upper bound of the set $\displaystyle X = \bigcup_{\substack{b \in L \\ b \prec c}} \mathrm{id}(b) \cap \mathrm{ji}(L)$. If $u$ is another upper bound of $X$, then $u$ is an upper bound of $\mathrm{id}(x) \cap \mathrm{ji}(L)$ for every $x \in L$ with $x \prec c$. As $c$ is minimal in $S$, then $\displaystyle x = \bigvee_{b \, \in \, \mathrm{id}(x) \cap \mathrm{ji}(L)} b$ if $x \prec c$, hence $u$ is an upper bound of $\{b \in L\ |\ b \prec c\}$ implying $u \succeq c$. Observe that
	$\displaystyle X = \bigcup_{\substack{b \in L \\ b \prec c}}\big(\mathrm{id}(b) \cap \mathrm{ji}(L)\big) = \mathrm{ji}(L) \cap \bigcup_{\substack{b \in L \\ b \prec c}} \mathrm{id}(b) = \mathrm{ji}(L) \cap \mathrm{id}(c)$. Therefore, $c$ is a minimal upper bound for $\mathrm{id}(c) \cap \mathrm{ji}(L)$ which is a contradiction.
\end{proof}

\noindent For two elements $a,b$ of a lattice $L$ such that $a \preceq b$, let $\mathrm{j}_{a,b}: [a \wedge b,\, b] \rightarrow [a,\, a \vee b]$ and $\mathrm{m}_{a,b}: [a,\, a \vee b] \rightarrow [a \wedge b,\, b]$ be functions respectively defined by $$\mathrm{j}_{a,b}(x):= a \vee x \quad \text{and} \quad \mathrm{m}_{a,b}(x):= x \wedge b.$$

\begin{definition}
A lattice $L$ is \textbf{modular} if, for all $a,b \in L$, $x \in [a \wedge b,\, b]$, and $y \in [a,\, a \vee b]$, we have $$x = \mathrm{m}_{a,b}\, \mathrm{j}_{a,b}(x) \quad \text{and} \quad y = \mathrm{j}_{a,b}\, \mathrm{m}_{a,b}(y).$$
\end{definition}

\begin{proposition} \label{PrMo}
A lattice $L$ is modular if and only if, for all $a,b,z \in L$, we have
$$(a \vee z) \wedge (a \vee b) = a \vee \big(z \wedge (a \vee b)\big) \quad \text{and} \quad (a \wedge z) \vee (a \wedge b) = a \wedge \big(z \vee (a \wedge b)\big).$$
\end{proposition}

\begin{proof}
Assume first that $L$ is modular. We have $a \preceq (a \vee z) \wedge (a \vee b) \preceq a \vee b$. Letting $u = (a \vee z) \wedge (a \vee b)$, we get $$u = \mathrm{j}_{a,b}\, \mathrm{m}_{a,b}(u) = a \vee \big((a \vee z) \wedge (a \vee b) \wedge b\big) = a \vee \big((a \vee z) \wedge b\big).$$
Since it is true for all $a,b,z \in L$, interchanging $z$ and $b$, we obtain $u = a \vee \big(z \wedge (a \vee b)\big)$. Likewise, we have $a \wedge b \preceq (z \wedge b) \vee (a \wedge b) \preceq b$. Letting $v = (z \wedge b) \vee (a \wedge b)$, we get $$v = \mathrm{m}_{a,b}\, \mathrm{j}_{a,b}(v) = b \wedge \big((z \wedge b) \vee (a \wedge b) \vee a\big) = b \wedge \big((b \wedge z) \vee a\big).$$
Since it is true for all $a,b,z \in L$, interchanging $z$ and $a$, we obtain $v = b \wedge \big(z \vee (a \wedge b)\big)$.

\smallskip

\noindent Assume now that $(a \vee z) \wedge (a \vee b) = a \vee \big(z \wedge (a \vee b)\big)$ and $(a \wedge z) \vee (a \wedge b) = a \wedge \big(z \vee (a \wedge b)\big)$ for all $a,b,z \in L$. If $a \preceq z \preceq a \vee b$, then
$$z = (a \vee b) \wedge z = (a \vee b) \wedge (a \vee z) = a \vee \big(b \wedge (a \vee z)\big).$$
Since $a \vee z = z$, then $z = a \vee (z \wedge b) = \mathrm{j}_{a,b}\, \mathrm{m}_{a,b}(z)$. Likewise, if $a \wedge b \preceq z \preceq b$, then
$$z = (a \wedge b) \vee z = (b \wedge a) \vee (z \wedge b) = b \wedge \big(a \vee (z \wedge b)\big).$$
And since $z \wedge b = z$, then $z = b \vee (a \wedge z) = \mathrm{m}_{a,b}\, \mathrm{j}_{a,b}(z)$.
\end{proof}

\subsection{Distributive Lattice}

\begin{proposition}
	Let $L$ be a poset, and $a,b,c \in L$. The condition
	$$a \wedge (b \vee c) = (a \wedge b) \vee (a \wedge c) \text{\: is equivalent to\: } a \vee (b \wedge c) = (a \vee b) \wedge (a \vee c).$$
\end{proposition}

\begin{proof}
	Assume that $a \wedge (b \vee c) = (a \wedge b) \vee (a \wedge c)$. Then,
	\begin{align*}
		(a \vee b) \wedge (a \vee c) & = \big((a \vee b) \wedge a\big) \vee \big((a \vee b) \wedge c\big) \\
		& = a \vee \big((a \vee b) \wedge c\big) \\
		& = a \vee (a \wedge c) \vee (b \wedge c) \\
		& = a \vee (b \wedge c).
	\end{align*}
	Similarly, if we assume $a \vee (b \wedge c) = (a \vee b) \wedge (a \vee c)$, then we obtain
	$$(a \wedge b) \vee (a \wedge c) = \big((a \wedge b) \vee a\big) \wedge \big((a \wedge b) \vee c\big) = a \wedge (b \vee c).$$
\end{proof}

\begin{definition}
	A lattice $L$ is \textbf{distributive} if, for all $a,b,c \in L$, $a \wedge (b \vee c) = (a \wedge b) \vee (a \wedge c)$.
\end{definition}

\noindent Denote by $\mathcal{I}_L$ the poset formed by the ideals of a lattice $L$ with inclusion as partial order. It is a lattice such that, for $I,J \in \mathcal{I}_L$, $\displaystyle I \vee J := \bigcap_{\substack{K \in \mathcal{I}_L \\ I,J \subseteq K}}K$ and $\displaystyle I \wedge J := \bigcup_{\substack{K \in \mathcal{I}_L \\ K \subseteq I \cap J}}K$.

\begin{theorem} \label{ThPr}
Let $L$ be a distributive lattice, $I$ an ideal of $L$, and $F$ a filter of $L$ such that $I \cap F = \emptyset$. Then, there exists a prime ideal $P$ of $L$ such that $I \subseteq P$ and $P \cap F = \emptyset$.
\end{theorem}

\begin{proof}
Set $\displaystyle \mathcal{X}_{I,F} := \{M \in \mathcal{I}_L\ |\ I \subseteq M,\, M \cap F = \emptyset\}$. It is a poset with inclusion as partial order, and is nonempty since $I \in \mathcal{X}_{I,F}$. Consider a chain $\mathcal{E} \in \mathcal{C}_{\mathcal{X}_{I,F}}$, and let $\displaystyle E = \bigcup_{C \in \mathcal{E}} E$. If $a,b \in E$, then $a \in A$ and $b \in B$ for some $A,B \in \mathcal{E}$. Since $\mathcal{E}$ is a chain, either $A \subseteq B$ or $A \supseteq B$ hold, so let assume $A \subseteq B$. Then, $a \in B$, and $a \vee b \in B \subseteq E$, as $B$ is an ideal. Moreover, if $c \in L$, then $a \wedge c \in A \subseteq E$, as $A$ is also an ideal. We deduce that $E \in \mathcal{I}_L$. Besides, $I \subseteq E$ and $E \cap F = \emptyset_{\prec}$ obviously hold. Hence, $E$ is an upper bound of $\mathcal{E}$ in $\mathcal{X}_{I,F}$. Therefore, $\mathcal{X}_{I,F}$ is an inductive poset, and Zorn's lemma allows to state that it has a maximal element $P$.

\noindent Suppose that $P$ is not prime. Then, there exists $a,b \in L$ such that $a,b \notin P$ but $a \wedge b \in P$. The maximality of $P$ yields $\big(P \vee \mathrm{id}(a)\big) \cap F \neq \emptyset$ and $\big(P \vee \mathrm{id}(b)\big) \cap F \neq \emptyset$. Thus, there are $p,q \in P$ such that $p \vee a \in F$, $q \vee b \in F$, and $(p \vee a) \wedge (q \vee b) \in F$ since $F$ is a filter. Expanding by distributivity, we obtain
$$(p \vee a) \wedge (q \vee b) = \big((p \vee a) \wedge q\big) \vee \big((p \vee a) \wedge b\big) = (p \wedge q) \vee (a \wedge q) \vee (p \wedge b) \vee (a \wedge b)$$
which belongs to $P$. That means $P \cap F \neq \emptyset$ or a contradiction. 
\end{proof}

\begin{corollary} \label{CoPr}
Let $L$ be a distributive lattice, $I \in \mathcal{I}_L$, and $a \in L$ such that $a \notin I$. Then, there exists a prime ideal $P$ of $L$ such that $I \subseteq P$ and $a \notin P$.
\end{corollary}

\begin{proof}
Remark that $I \neq \mathrm{fil}(a) = \emptyset$, otherwise, if $b \in I \neq \mathrm{fil}(a)$, then $b \wedge a = a \in I$, which is absurd. Now, for the proof, we apply Theorem~\ref{ThPr} to $I$ and $F = \mathrm{fil}(a)$.
\end{proof}

\begin{corollary} \label{CoSep}
Let $L$ be a distributive lattice, and $a,b \in L$ such that $a \neq b$. Then, $L$ has a prime ideal containing exactly one of $a$ and $b$.
\end{corollary}

\begin{proof}
If $a$ and $b$ are not comparable or $b \prec a$, then $a \notin \mathrm{id}(b)$. It remains to apply Corollary~\ref{CoPr} to $I = \mathrm{id}(b)$.
\end{proof}

\begin{theorem} \label{ThEq}
A lattice $L$ is distributive if and only if, for all $a,b,c \in L$, $c \vee a = c \vee b$ and $c \wedge a = c \wedge b$ imply $a = b$.
\end{theorem}

\begin{proof}
Suppose first that $L$ is distributive and that there exist $a,b,c \in L$ such that $a \vee c = b \vee c$ and $a \wedge c = b \wedge c$. Then,
$$a = a \vee (a \wedge c) = a \vee (b \wedge c) = (a \vee b) \wedge (a \vee c) = (a \vee b) \wedge (b \vee c) = b \vee (a \wedge c),$$
which implies $a \preceq b$, and similarly we have $b \preceq a$.

\smallskip

\noindent Suppose now that $a \vee c = b \vee c$ and $a \wedge c = b \wedge c$ imply $a = b$. If $x \in [a \wedge b,\,b]$,
\begin{itemize}
\item as $x \preceq b \wedge (a \vee x)$ then $a \vee x \preceq a \vee \big(b \wedge (a \vee x)\big)$, as $a \vee x \succeq b \wedge (a \vee x)$ then $a \vee x \succeq a \vee \big(b \wedge (a \vee x)\big)$, hence $a \vee x = a \vee \big(b \wedge (a \vee x)\big)$ on one side,
\item on the other side, $a \wedge x = a \wedge b \wedge x = a \wedge b \wedge (a \vee x)$.
\end{itemize}
By canceling $a$, we obtain $x = \mathrm{m}_{a,b}\, \mathrm{j}_{a,b}(x)$. If $y \in [a,\,a \vee b]$, as $y \succeq a \vee (b \wedge y)$ then $b \wedge y \succeq b \wedge \big(a \vee (b \wedge y)\big)$, as $b \wedge y \preceq a \vee (b \wedge y)$ then $b \wedge y \preceq b \wedge \big(a \vee (b \wedge y)\big)$, hence $b \wedge y = b \wedge \big(a \vee (b \wedge y)\big)$ on one side, and $b \vee y = a \vee b \vee y = a \vee b \vee (b \wedge y)$ on the other side. By canceling $b$, we obtain $y = \mathrm{j}_{a,b}\, \mathrm{m}_{a,b}(y)$. Therefore, $L$ is modular.

\noindent Let $a^* = a \wedge (b \vee c)$, $b^* = b \wedge (c \vee a)$, and $c^* = c \wedge (a \vee b)$. Then, $a^* \wedge b^* = a \wedge (c \vee a) \wedge b \wedge (b \vee c) = a \wedge b$, $a^* \wedge c^* = a \wedge c$, and $b^* \wedge c^* = b \wedge c$. Set $d = (a \vee b) \wedge (b \vee c) \wedge (c \vee a)$. Using twice Proposition~\ref{PrMo}, we get
\begin{align*}
a^* \vee b^* & = a^* \vee \big(b \wedge (a \vee c)\big) = (a^* \vee b) \wedge (a \vee c) \\
& = \Big(\big((b \vee c) \wedge a\big) \vee b\Big) \wedge (a \vee c) = (b \vee c) \wedge (a \vee b) \wedge (a \vee c) \\
& = d.
\end{align*}
By symmetry, we also have $a^* \vee c^* = b^* \vee c^* = d$. Hence,
\begin{itemize}
\item $c^* \vee a^* \vee (b \wedge c) = c^* \vee b^* \vee (a \wedge c) = d$,
\item and $c^* \wedge \big(a^* \vee (b \wedge c)\big) = (c^* \wedge a^*) \vee (b \wedge c) = (c^* \wedge b^*) \vee (a \wedge c) = c^* \wedge \big(b^* \vee (a \wedge c)\big)$.
\end{itemize}
By canceling $c^*$, we obtain $a^* \vee (b \wedge c) = b^* \vee (a \wedge c)$, whence
$$a^* \vee (b \wedge c) = a^* \vee (b \wedge c) \vee b^* \vee (a \wedge c) = a^* \vee b^* = d.$$
It follows that $(a \vee b) \wedge c = c^* = c^* \wedge d = c^* \wedge \big(a^* \vee (b \wedge c)\big) = (a \wedge c) \vee (b \wedge c)$, hence $L$ is consequently distributive.
\end{proof}

\section{Valuation on Lattice} \label{SecVa}

\noindent This section is the central part of this survey. After defining the valuation algebra and showing some important properties, we prove that if $M$ is a subset of a complete lowly finite distributive lattice $L$ containing its join-irreducible elements, and $a$ an element of $M$ which is not join-irreducible, then $\displaystyle \sum_{b \in M \cap [\mathsf{0},a]} \mu_M(b,a)b$ belongs to the submodule $\langle a \wedge b + a \vee b - a - b\ |\ a,b \in L \rangle$ of $\mathbb{Z}L$. It would not have been possible to write the first two subsections without the articles of Geissinger \cite{Ge1}, \cite[§~3]{Ge2}, and the third without that of Zaslavsky \cite[§~2]{Za}.

\begin{definition}
A \textbf{valuation} on a lattice $L$ is a function $f$ from $L$ to a module $G$ such that, for all $a,b \in L$,
$$f(a \wedge b) + f(a \vee b) = f(a) + f(b).$$ 
\end{definition}

\subsection{Valuation Module}

\begin{definition}
The \textbf{valuation module} of a lattice $L$ is the module $\mathrm{Val}(L) := \mathbb{Z}L/\mathrm{N}(L)$, where $\mathrm{N}(L)$ is the submodule $\langle a \wedge b + a \vee b - a - b\ |\ a,b \in L \rangle$ of the module $\mathbb{Z}L$.
\end{definition}

\begin{proposition} \label{PrHom}
Let $\mathrm{i}:L \rightarrow \mathrm{Val}(L)$ be the natural induced map for a lattice $L$. Then, $\mathrm{i}$ is a valuation, and, for every valuation $f:L \rightarrow G$, there exists a unique module homomorphism $h: \mathrm{Val}(L) \rightarrow G$ such that the following diagram is commutative 
\begin{center}
	\begin{tikzcd}
		L \arrow[rd, "f"] \arrow[r, "\mathrm{i}"] & \mathrm{Val}(L) \arrow[d, "h"] \\
		& G
	\end{tikzcd}
\end{center} 
\end{proposition}

\begin{proof}
It is clear that $\mathrm{i}$ is a valuation as $\mathrm{i}(a \wedge b) + \mathrm{i}(a \vee b) - \mathrm{i}(a) - \mathrm{i}(b) = a \wedge b + a \vee b - a - b = 0$. Besides, we get the homomorphism $h$ by setting
$$\forall a \in L:\, h(a) := f(a) \quad \text{and} \quad \forall x,y \in \mathrm{Val}(L):\, h(x+y) = h(x) + h(y).$$
\end{proof}

\begin{proposition}
For lattices $L_1, L_2$ with natural induced maps $\mathrm{i}_1, \mathrm{i}_2$ respectively, a lattice homomorphism $\varphi: L_1 \rightarrow L_2$ induces a unique module homomorphism $\psi: \mathrm{Val}(L_1) \rightarrow \mathrm{Val}(L_2)$ such that, for every $a \in L_1$, $\psi \mathrm{i}_1(a) = \mathrm{i}_2 \varphi(a)$.
\end{proposition}

\begin{proof}
We obtain the homomorphism $\psi$ by setting
$$\forall a \in L_1:\, \psi(a) := \varphi(a) \quad \text{and} \quad \forall x,y \in \mathrm{Val}(L_1):\, \psi(x+y) = \psi(x) + \psi(y).$$
\end{proof}

\begin{proposition} \label{PrIn}
For any prime ideal or prime filter $M$ of a lattice $L$ with natural induced map $\mathrm{i}$, each element of $\mathrm{i}(M)$ is linearly independent of those in $\mathrm{i}(L \setminus M)$ and vise versa. 
\end{proposition}

\begin{proof}
Assume that $M$ is a prime ideal, and consider the indicator function $\mathrm{1}_M:L \rightarrow \mathbb{Z}$ defined as $\mathrm{1}_M(a) := \begin{cases}
	1 & \text{if } a \in M \\
	0 & \text{otherwise} 
\end{cases}$. For $a,b \in L$,
\begin{itemize}
\item if $a,b \in M$, we clearly have $\mathrm{1}_M(a \wedge b) + \mathrm{1}_M(a \vee b) = \mathrm{1}_M(a) + \mathrm{1}_M(b) = 2$,
\item if $a \in M$ and $b \notin M$, since $(a \vee b) \wedge b = b \notin M$, then $a \vee b \notin M$ and $\mathrm{1}_M(a \wedge b) + \mathrm{1}_M(a \vee b) = \mathrm{1}_M(a) + \mathrm{1}_M(b) = 1$,
\item if $a,b \notin M$, then $a \wedge b \notin M$, the fact $(a \vee b) \wedge b = b \notin M$ implies $a \vee b \notin M$, and $\mathrm{1}_M(a \wedge b) + \mathrm{1}_M(a \vee b) = \mathrm{1}_M(a) + \mathrm{1}_M(b) = 0$.
\end{itemize}
Therefore, $\mathrm{1}_M$ is a valuation on $L$. One similarly proves that if $M$ is prime filter, then $\mathrm{1}_M$ is also a valuation on $L$. We know from Proposition~\ref{PrHom} that there exists a unique homomorphism $h: \mathrm{Val}(L) \rightarrow \mathbb{Z}$ such that the diagram
\begin{tikzcd}
		L \arrow[rd, "\mathrm{1}_M"] \arrow[r, "\mathrm{i}"] & \mathrm{Val}(L) \arrow[d, "h"] \\
		& \mathbb{Z}
\end{tikzcd}
is commutative. As $h\mathrm{i}(a) = 1$, for every $a \in M$, and $\big\langle \mathrm{i}(b)\ \big|\ b \in L \setminus M \big\rangle \subseteq \ker h$, each element of $\mathrm{i}(M)$ is then linearly independent of those in $\mathrm{i}(L \setminus M)$. Likewise, Proposition~\ref{LeDu} allows to state that $\mathrm{1}_{L \setminus M}$ is a valuation, then one also proves that each element of $\mathrm{i}(L \setminus M)$ is linearly independent of those in $\mathrm{i}(M)$. 
\end{proof}

\begin{proposition} \label{PrInj}
The natural induced map $\mathrm{i}:L \rightarrow \mathrm{Val}(L)$ of a lattice $L$ is an injection if and only if $L$ is distributive. 
\end{proposition}

\begin{proof}
If $L$ is distributive, we know from Corollary~\ref{CoSep} that any two different elements $a,b \in L$ can be separated by a prime ideal, hence Proposition~\ref{PrIn} allows to deduce that $\mathrm{i}(a)$ and $\mathrm{i}(b)$ are independent in $\mathrm{Val}(L)$.

\noindent If $L$ is not distributive, then, by Theorem~\ref{ThEq}, it contains distinct elements $a,b,c$ with $c \vee a = c \vee b$ and $c \wedge a = c \wedge b$. Hence, $\mathrm{i}(a) + \mathrm{i}(c) = \mathrm{i}(c \vee a) + \mathrm{i}(c \wedge a) = \mathrm{i}(c \vee b) + \mathrm{i}(c \wedge b) = \mathrm{i}(b) + \mathrm{i}(c)$, and $\mathrm{i}(a) = \mathrm{i}(b)$.
\end{proof}

\begin{proposition} \label{PrMDi}
	Let $L$ be a distributive lattice, and $a_1, \dots, a_n, b \in L$ with $\displaystyle b \notin \big[\bigwedge_{i \in [n]}a_i,\, \bigvee_{i \in [n]}a_i\big]$. Then, $b$ is linearly independent of $\{a_1, \dots, a_n\}$ in $\mathrm{Val}(L)$.
\end{proposition}

\begin{proof}
	If $\displaystyle b \notin \mathrm{id}\Big(\bigvee_{i \in [n]}a_i\Big)$, then there exists a prime ideal $P$ such that $\{a_1, \dots, a_n\} \subseteq P$ and $b \notin P$ by Corollary~\ref{CoPr}, and $b$ is linearly independent of $\{a_1, \dots, a_n\}$ by Proposition~\ref{PrIn}.
	
	\noindent If $\displaystyle b \in \mathrm{id}\Big(\bigvee_{i \in [n]}a_i\Big)$, then $\displaystyle b \notin \mathrm{fil}\Big(\bigwedge_{i \in [n]}a_i\Big)$, otherwise $\displaystyle b \in \big[\bigwedge_{i \in [n]}a_i,\, \bigvee_{i \in [n]}a_i\big]$ which is a contradiction. Hence, $\displaystyle \mathrm{id}(b) \cap \mathrm{fil}\Big(\bigwedge_{i \in [n]}a_i\Big) = \emptyset$, and there exists a prime ideal $P$ such that $\mathrm{id}(b) \subseteq P$ and $\displaystyle P \cap \mathrm{fil}\Big(\bigwedge_{i \in [n]}a_i\Big) = \emptyset$ by Theorem~\ref{ThPr}. As $\displaystyle \{a_1, \dots, a_n\} \subseteq \mathrm{fil}\Big(\bigwedge_{i \in [n]}a_i\Big)$, we once again obtain the independence of $b$ by Proposition~\ref{PrIn}.
\end{proof}

\noindent As the lattice $L$ with either the operation $\vee$ or $\wedge$ form a semigroup, the module $\mathbb{Z}L$ may consequently be considered as an algebra with either $\vee$ or $\wedge$ as multiplication. Besides, if $L$ is distributive, Proposition~\ref{PrInj} allows to identify $L$ with $\mathrm{i}(L)$.

\begin{proposition} \label{PrHL}
If $L$ is a distributive lattice, then $\mathrm{N}(L)$ is an ideal of the algebra $\mathbb{Z}L$ for both $\vee$ and $\wedge$ as multiplication.
\end{proposition}

\begin{proof}
For $a,b,c \in L$, we have
\begin{align*}
(a \wedge b + a \vee b - a - b) \wedge c & = (a \wedge b) \wedge c + (a \vee b) \wedge c - a \wedge c - b \wedge c \\
& = (a \wedge c) \wedge (b \wedge c) + (a \wedge c) \vee (b \wedge c) - a \wedge c - b \wedge c
\end{align*}
which belongs to $\mathrm{N}(L)$. Then, by linearly extension, we get $(a \wedge b + a \vee b - a - b) \wedge t \in \mathrm{N}(L)$ for any $t \in \mathbb{Z}L$. Similarly, we have $$(a \wedge b + a \vee b - a - b) \vee c = (a \vee c) \wedge (b \vee c) + (a \vee c) \vee (b \vee c) - a \vee c - b \vee c \in \mathrm{N}(L).$$ 
\end{proof}

\subsection{Valuation Algebra}

\noindent If the lattice $L$ is distributive, Proposition~\ref{PrHL} allows to state that the valuation module $\mathrm{Val}(L)$ becomes a commutative algebra for either $\vee$ or $\wedge$ as multiplication.

\begin{definition}
The $\textbf{valuation algebra}$ is the algebra $\big(\mathrm{Val}(L), \vee\big)$ or $\big(\mathrm{Val}(L), \wedge\big)$ for a distributive lattice $L$.
\end{definition}

\begin{lemma} \label{LeTau}
Let $L$ be a complete distributive lattice, and define the map $\tau: \mathrm{Val}(L) \rightarrow \mathrm{Val}(L)$ by $\tau(x):= \mathsf{1} + \mathsf{0} - x$. Then, for $a,b \in L$, we have $\tau(a \vee b) = \tau(a) \wedge \tau(b)$.
\end{lemma}

\begin{proof}
We have $\mathsf{1} + \mathsf{0} - a \vee b = \mathsf{1} + \mathsf{0} + a \wedge b - a - b = (\mathsf{1} + \mathsf{0} - a) \wedge (\mathsf{1} + \mathsf{0} - b)$.
\end{proof}

\begin{proposition} \label{PrDi}
Let $L$ be a complete distributive lattice, $n \in \mathbb{N}^*$, and $a_1, \dots, a_n \in L$. Then, we have $\displaystyle \mathsf{1} - \bigvee_{i \in [n]}a_i = \bigwedge_{i \in [n]}(\mathsf{1} - a_i)$, that is $$\displaystyle \bigvee_{i \in [n]}a_i = \sum_{k=1}^n (-1)^{k-1} \sum_{\substack{I \subseteq [n] \\ \#I=k}} \bigwedge_{i \in I}a_i.$$
\end{proposition}

\begin{proof}
Using Lemma~\ref{LeTau} and $\mathsf{0} \wedge (\mathsf{1} - a_i) = \mathsf{0}$, we obtain
$$\tau\Big(\bigvee_{i \in [n]}a_i\Big) = \mathsf{0} + \mathsf{1} - \bigvee_{i \in [n]}a_i = \bigwedge_{i \in [n]} \tau(a_i) = \bigwedge_{i \in [n]}(\mathsf{0} + \mathsf{1} - a_i) = \mathsf{0} + \bigwedge_{i \in [n]}(\mathsf{1} - a_i).$$
Then $\displaystyle \mathsf{1} - \bigvee_{i \in [n]}a_i = \bigwedge_{i \in [n]} \tau(a_i) = \bigwedge_{i \in [n]}(\mathsf{1} - a_i) = \mathsf{1} + \sum_{k=1}^n (-1)^k \sum_{\substack{I \subseteq [n] \\ \#I=k}} \bigwedge_{i \in I}a_i$. 
\end{proof}

\begin{corollary} \label{CoDi}
Let $L$ be a complete distributive lattice, $n \in \mathbb{N}^*$, $a_1, \dots, a_n \in L$, and $f$ a valuation on $L$. Then,
$$f\Big(\bigvee_{i \in [n]}a_i\Big) = \sum_{k=1}^n (-1)^{k-1} \sum_{\substack{I \subseteq [n] \\ \#I=k}} f\Big(\bigwedge_{i \in I}a_i\Big).$$
\end{corollary}

\begin{proof}
If $f$ is a valuation to module $G$, we know from Proposition~\ref{PrHom} that a unique module homomorphism $h:\mathrm{Val}(L) \rightarrow G$ such that $h\mathrm{i}=f$ exists. Then, using Proposition~\ref{PrDi}, we obtain
$$f\Big(\bigvee_{i \in [n]}a_i\Big) = h\Big(\bigvee_{i \in [n]}a_i\Big) = \sum_{k=1}^n (-1)^{k-1} \sum_{\substack{I \subseteq [n] \\ \#I=k}} h\Big(\bigwedge_{i \in I}a_i\Big) = \sum_{k=1}^n (-1)^{k-1} \sum_{\substack{I \subseteq [n] \\ \#I=k}} f\Big(\bigwedge_{i \in I}a_i\Big).$$
\end{proof}

\begin{theorem} \label{ThAb}
Let $L$ be a complete lowly finite distributive lattice. Then, $\mathrm{Val}(L)$ is equal to $\mathbb{Z}\mathrm{ji}(L)$ as modules. 
\end{theorem}

\begin{proof}
We obviously have $\mathsf{0} \in \mathrm{ji}(L)$. Let $a \in L$, and assume that every $b \in L$ such that $a \succ b$ is a linear combination in $\mathrm{Val}(L)$ of a finite number of elements in $\mathrm{ji}(L)$. We know from Proposition~\ref{PrJi} that there exists a subset $\{b_1, \dots, b_n\}$ of $\mathrm{ji}(L)$ such that $\displaystyle a = \bigvee_{i \in [n]}b_i$. Using Proposition~\ref{PrDi}, we get $\displaystyle a = \sum_{k=1}^n (-1)^{k-1} \sum_{\substack{I \subseteq [n] \\ \#I=k}} \bigwedge_{i \in I}b_i$ with $\displaystyle a \succ \bigwedge_{i \in I}b_i$ for each $I \subseteq [n]$. Thus $\mathrm{ji}(L)$ generates $\mathrm{Val}(L)$.

\noindent Assume now that every subset with cardinality $n-1$ in $\mathrm{ji}(L)$ is independent, and consider a subset of $n$ elements $\{a_1, \dots, a_n\} \subseteq \mathrm{ji}(L)$. We can suppose that $a_n$ is a maximal element in that set. Since $\displaystyle a_n \neq \bigvee_{i \in [n-1]}a_i$, then $\displaystyle a_n \notin \big[\bigwedge_{i \in [n-1]}a_i,\, \bigvee_{i \in [n-1]}a_i\big]$. We deduce from Proposition~\ref{PrMDi} that $\{a_1, \dots, a_n\}$ is independent. Hence $\mathrm{ji}(L)$ is an independent set in $\mathrm{Val}(L)$.
\end{proof}

\begin{corollary}
If $L$ is a complete lowly finite distributive lattice, then every valuation of $L$ is determined by its values on $\mathrm{ji}(L)$ which can be assigned arbitrarily.
\end{corollary}

\begin{proof}
If $f$ is a valuation to a module $G$, we know from Proposition~\ref{PrHom} that a unique module homomorphism $h:\mathrm{Val}(L) \rightarrow G$ such that $h\mathrm{i}=f$ exists. We know from Theorem~\ref{ThAb} that, if $a \in L$, there exist subsets $\{\lambda_1, \dots, \lambda_n\} \subseteq \mathbb{Z}$ and $\{a_1, \dots, a_n\} \subseteq \mathrm{ji}(L)$ such that $\displaystyle a = \sum_{i \in [n]} \lambda_i a_i$. Then, $\displaystyle f(a) = h(a) = h\Big(\sum_{i \in [n]} \lambda_i a_i\Big) = \sum_{i \in [n]} \lambda_i h(a_i) = \sum_{i \in [n]} \lambda_i f(a_i)$.
\end{proof}

\noindent For a poset $L$, and $a,b \in L$, we write $a \lessdot b$ if $a \prec b$ and $\{c \in L\ |\ a \prec c \prec b\} = \emptyset$.

\begin{proposition}
Let $L$ be a distributive lattice, and $a \in \mathrm{ji}(L)$ such that $a$ is not minimal. Then, there exists a unique element $a^* \in L$ such that $a^* \lessdot a$.
\end{proposition}

\begin{proof}
Suppose that there exist two different elements $b,c \in L$ such that $b \lessdot a$ and $c \lessdot a$. Then, $b \vee c \succeq b$, $b \vee c \succeq c$, and $b \vee c \notin \{b,c\}$. The only possibility is $b \vee c = a$ which contradicts the join-irreducibility of $a$.
\end{proof}

\noindent Let $L$ be a distributive lattice having a lowest element $\mathsf{0}$. Define $$e_{\mathsf{0}} := \mathsf{0} \in \mathrm{Val}(L) \quad \text{and} \quad e_a := a - a^* \in \mathrm{Val}(L)\ \text{for each}\ a \in \mathrm{ji}(L) \setminus \{\mathsf{0}\}.$$

\begin{theorem}
Let $L$ be a complete lowly finite distributive lattice. Then, $\big\{e_a\ |\ a \in \mathrm{ji}(L)\big\}$ is an orthogonal idempotent basis of $\mathrm{Val}(L)$.
\end{theorem}

\begin{proof}
For $a,b \in \mathrm{ji}(L) \setminus \{\mathsf{0}\}$ with $a \neq b$, we have $e_{\mathsf{0}} \wedge e_{\mathsf{0}} = e_{\mathsf{0}}$ and $e_a \wedge e_{\mathsf{0}} = a \wedge \mathsf{0} - a^* \wedge \mathsf{0} = 0$,
$e_a \wedge e_a = a \wedge a - a \wedge a^* - a^* \wedge a + a^* \wedge a^* = a - a^* - a^* + a^* = e_a$, and \begin{align*}
	e_a \wedge e_b & = a \wedge b - a \wedge b^* - a^* \wedge b + a^* \wedge b^* \\
	& = \begin{cases}
		b - b^* - b + b^* & \text{if}\ a^* = b \\
		a^* \wedge b^* - a^* \wedge b^* - a^* \wedge b^* + a^* \wedge b^* & \text{otherwise}
	\end{cases} \\
	& = 0.
\end{align*}
Then, $\big\{e_a\ |\ a \in \mathrm{ji}(L)\big\}$ is orthogonal idempotent. Assume now that every subset with cardinality $n-1$ in $\big\{e_a\ |\ a \in \mathrm{ji}(L)\big\}$ is independent, and consider a subset of $n$ elements $\{e_{a_1}, \dots, e_{a_n}\}$. We can suppose that $a_n$ is a maximal element in the set $\{a_1, \dots, a_n\}$. Since $\displaystyle a_n \neq \bigvee_{i \in [n-1]}a_i  \vee \bigvee_{i \in [n]}a_{i}^*$, then $\displaystyle a_n \notin \big[\bigwedge_{i \in [n-1]}a_i \wedge \bigwedge_{i \in [n]}a_{i}^*,\, \bigvee_{i \in [n-1]}a_i \vee \bigvee_{i \in [n]}a_{i}^*\big]$. We deduce from Proposition~\ref{PrMDi} that $a_n$ is independent of $\{a_1, \dots, a_{n-1}, a_{1}^*, \dots, a_{n}^*\}$. Hence $e_{a_n}$ is independent of $\{e_{a_1}, \dots, e_{a_{n-1}}\}$, and $\{e_{a_1}, \dots, e_{a_n}\}$ is consequently an independent set in $\mathrm{Val}(L)$. Finally, since there is a natural bijection $a \mapsto e_a$ between $\mathrm{ji}(L)$ and $\big\{e_a\ |\ a \in \mathrm{ji}(L)\big\}$, by Theorem~\ref{ThAb} the latter is also a basis of $\mathrm{Val}(L)$.
\end{proof}

\subsection{Identities on Valuation Algebra}

\begin{theorem} \label{ThPrMo}
Let $L$ be a complete lowly finite distributive lattice. Then,
$$\forall x \in L:\, x = \sum_{\substack{a,b \in \mathrm{ji}(L) \\ b \preceq a \preceq x}} \mu_{\mathrm{ji}(L)}(b,a)b.$$
\end{theorem}

\begin{proof}
If $a \in \mathrm{ji}(L)$, then $a = e_a + a^*$, particularly $\mathsf{0} = e_{\mathsf{0}}$. Now, consider any $x \in L \setminus \mathrm{ji}(L)$, and assume that, for every $b \in L$ such that $b \prec x$, we have $\displaystyle b = \sum_{\substack{d \in \mathrm{ji}(L) \\ d \preceq b}} e_d$. There exist $b,c \in L \setminus \{x\}$ such that $x = b \vee c$. Note that $\displaystyle b \wedge c = \sum_{\substack{d \in \mathrm{ji}(L) \\ d \preceq b \wedge c}}e_d$ as the $e_d$'s are orthogonal idempotent. Hence, $\displaystyle b \vee c = b + c - b \wedge c = \sum_{d\, \in \, \mathrm{ji}(L) \cap (\mathrm{id}(b) \cup \mathrm{id}(c))}e_d$. Besides, remark that, for any $y \in \mathrm{ji}(L) \cap \mathrm{id}(x)$, there exist $b,c \in L \setminus \{x\}$ such that $y \preceq b$ and $b \vee c = x$. Therefore, $\displaystyle x = \sum_{\substack{d \in \mathrm{ji}(L) \\ d \preceq x}}e_d$. 

\noindent Let $\mathrm{b}$ be the natural bijection $a \mapsto e_a$ between $\mathrm{ji}(L)$ and $\big\{e_a\ |\ a \in \mathrm{ji}(L)\big\}$. For $a \in \mathrm{ji}(L)$, we have $\displaystyle \mathrm{i}(a) = \sum_{d \, \in \, \mathrm{ji}(L) \cap [\mathsf{0},a]}\mathrm{b}(d)$. Then, using the Möbius inversion formula, we obtain $$\displaystyle \mathrm{b}(a) = \sum_{d \, \in \, \mathrm{ji}(L) \cap [\mathsf{0},a]} \mu_{\mathrm{ji}(L)}(d,a) \mathrm{i}(d) \quad \text{or} \quad e_a = \sum_{\substack{d \in \mathrm{ji}(L) \\ d \preceq a}} \mu_{\mathrm{ji}(L)}(d,a) d.$$

\noindent We obtain the result by combining $\displaystyle x = \sum_{\substack{a \in \mathrm{ji}(L) \\ a \preceq x}}e_a$ with $\displaystyle e_a = \sum_{\substack{d \in \mathrm{ji}(L) \\ d \preceq a}} \mu_{\mathrm{ji}(L)}(d,a) d$.
\end{proof}

\begin{lemma} \label{LeMob}
If $L$ is a lowly finite distributive lattice, then $\big(\mathbb{Z}L, \wedge\big)$ is naturally isomorphic to the Möbius algebra $\big(\mathrm{M\ddot{o}b}(L),\cdot\big)$.
\end{lemma}

\begin{proof}
For $a \in L$, we have $\displaystyle u_L(a) = \sum_{c \in [\mathsf{0},a]} \mu_L(c,a)c$. The Möbius inversion formula consequently allows to state that $\displaystyle a = \sum_{c \in [\mathsf{0},a]} u_L(c)$. Then, for $a,b \in L$, we have
$$a \cdot b = \sum_{c \in [\mathsf{0},a] \cap [\mathsf{0},b]} u_L(c) = \sum_{c \in [\mathsf{0},a \wedge b]} u_L(c) = a \wedge b.$$
\end{proof}

\begin{lemma} \label{LeIso}
If $L$ is a complete lowly finite distributive lattice, then $\big(\mathrm{M\ddot{o}b}(L)/\mathrm{N}(L),\cdot\big)$ is isomorphic to the Möbius algebra $\Big(\mathrm{M\ddot{o}b}\big(\mathrm{ji}(L)\big),\cdot\Big)$.
\end{lemma}

\begin{proof}
By Lemma~\ref{LeMob}, we get $\mathrm{M\ddot{o}b}(L)/\mathrm{N}(L) \simeq \mathbb{Z}L/\mathrm{N}(L) \simeq \mathrm{Val}(L)$. We know from Theorem~\ref{ThAb} that $\mathrm{Val}(L)$ is isomorphic to $\mathbb{Z}\mathrm{ji}(L)$ as modules. Now, as algebras, $\big(\mathrm{Val}(L), \wedge\big)$ is naturally isomorphic to $\Big(\mathrm{M\ddot{o}b}\big(\mathrm{ji}(L)\big),\cdot\Big)$ since, for $a,b \in \mathrm{ji}(L)$, Theorem~\ref{ThPrMo} allows to state that
$$a \cdot b = \sum_{c \in [\mathsf{0},a] \cap [\mathsf{0},b] \cap \mathrm{ji}(L)} u_{\mathrm{ji}(L)}(c) = \sum_{c \in [\mathsf{0},a \wedge b] \cap \mathrm{ji}(L)} u_{\mathrm{ji}(L)}(c) = a \wedge b.$$
\end{proof}

\noindent The following theorem is the main result of this survey. Zaslavsky originally proved it for finite distributive lattice \cite[Theorem~2.1]{Za}.

\begin{theorem} \label{ThUma}
Let $L$ be a complete lowly finite distributive lattice, and $M$ a subset of $L$ such that $\mathrm{ji}(L) \subseteq M$. If $a \in M \setminus \mathrm{ji}(L)$, then $$u_M(a) \in \mathrm{N}(L).$$
\end{theorem}

\begin{proof}
Consider the linear maps $\mathrm{j}: \mathrm{M\ddot{o}b}(L) \rightarrow \mathrm{M\ddot{o}b}\big(\mathrm{ji}(L)\big)$, $\mathrm{j}_1: \mathrm{M\ddot{o}b}(L) \rightarrow \mathrm{M\ddot{o}b}(M)$, and $\mathrm{j}_2: \mathrm{M\ddot{o}b}(M) \rightarrow \mathrm{M\ddot{o}b}\big(\mathrm{ji}(L)\big)$ which on the basis $\{u_L(a)\ |\ a \in L\}$, and $\{u_M(a)\ |\ a \in M\}$ respectively have the values
\begin{align*}
& \mathrm{j}\big(u_L(a)\big) := \begin{cases}
	u_{\mathrm{ji}(L)}(a) & \text{if}\ a \in \mathrm{ji}(L),\\
	0 & \text{otherwise}
\end{cases}, \quad \mathrm{j}_1\big(u_L(a)\big) := \begin{cases}
	u_M(a) & \text{if}\ a \in M,\\
	0 & \text{otherwise}
\end{cases}, \\
& \text{and}\ \mathrm{j}_2\big(u_M(a)\big) := \begin{cases}
	u_{\mathrm{ji}(L)}(a) & \text{if}\ a \in \mathrm{ji}(L),\\
	0 & \text{otherwise}
\end{cases}.
\end{align*}
Then, $\mathrm{j}$, $\mathrm{j}_1$, and $\mathrm{j}_2$ are algebra homomorphisms by Corollary~\ref{CoHoA}. Moreover, as the diagram
\begin{center}
\begin{tikzcd}
	\mathrm{M\ddot{o}b}(L) \arrow[rd, "\mathrm{j}"] \arrow[r, "\mathrm{j}_1"] & \mathrm{M\ddot{o}b}(M) \arrow[d, "\mathrm{j}_2"] \\
	& \mathrm{M\ddot{o}b}\big(\mathrm{ji}(L)\big)
\end{tikzcd}
\end{center}
is commutative, then $u_M(a) \in \ker \mathrm{j}_2 \subseteq \ker \mathrm{j}$ if $a \in M \setminus \mathrm{ji}(L)$.

\noindent Finally, since $\mathrm{M\ddot{o}b}\big(\mathrm{ji}(L)\big) \simeq \mathrm{M\ddot{o}b}(L)/\ker \mathrm{j}$ \cite[II-Theorem~6.12]{BuSa}, we obtain $\ker \mathrm{j} = \mathrm{N}(L)$ using Lemma~\ref{LeIso}, and consequently $u_M(a) \in \mathrm{N}(L)$ if $a \in M \setminus \mathrm{ji}(L)$. 
\end{proof}

\begin{corollary} \label{CoUma}
Let $L$ be a complete lowly finite distributive lattice, $M$ a subset of $L$ such that $\mathrm{ji}(L) \subseteq M$, and $f:L \rightarrow G$ a valuation on $L$. If $a \in M \setminus \mathrm{ji}(L)$, then 
$$\sum_{b \in [\mathsf{0},a] \cap M} \mu_M(b,a)f(b) = 0.$$
\end{corollary}

\begin{proof}
Let $h:\mathrm{Val}(L) \rightarrow G$ be the module homomorphism associated to $f$ as in Proposition~\ref{PrHom}. We already know from Lemma~\ref{LeMob} that $\mathrm{Val}(L) \simeq \mathrm{M\ddot{o}b}(L)/\mathrm{N}(L)$. By Theorem~\ref{ThUma}, we then obtain
\begin{align*}
\sum_{b \in [\mathsf{0},a] \cap M} \mu_M(b,a)b & = 0 \\
h\Big(\sum_{b \in [\mathsf{0},a] \cap M} \mu_M(b,a)b\Big) & = h(0) \\
\sum_{b \in [\mathsf{0},a] \cap M} \mu_M(b,a)h(b) & = 0 \\
\sum_{b \in [\mathsf{0},a] \cap M} \mu_M(b,a)f(b) & = 0.
\end{align*}
\end{proof}

\section{Dissection Theory} \label{SecDi}

\noindent We use Corollary~\ref{CoUma} to prove the fundamental theorem of dissection theory.

\begin{definition}
Let us call \textbf{subspace arrangement} in a topological space $T$ a finite set of subspaces in $T$.
\end{definition}

\noindent For a subspace arrangement $\mathscr{A}$ in $T$, let $\displaystyle L_{\mathscr{A}} := \Big\{\bigcap_{H \in \mathscr{B}}H \in 2^T \setminus \{\emptyset\}\ \Big|\ \mathscr{B} \subseteq \mathscr{A}\Big\}$ be the poset with partial order $\preceq$ defined, for $A,B \in L_{\mathscr{A}}$, by $A \preceq B$ if and only if $A \subseteq B$.

\begin{definition}
Let $\mathscr{A}$ be a subspace arrangement in a topological space $T$. A \textbf{meet-refinement} of $L_{\mathscr{A}}$ is a finite poset $L \subseteq 2^T \setminus \{\emptyset\}$ with the same partial order as that defined for $L_{\mathscr{A}}$ such that $\displaystyle \bigcup_{X \in L}X = \bigcup_{H \in \mathscr{A}}H$ and
\begin{itemize}
\item any element in $L_{\mathscr{A}}$ is a union of elements in $L$,
\item any nonempty intersection of elements in $L$ is also a union of elements in $L$.
\end{itemize}
\end{definition}

\noindent Denote by $\mathrm{C}(X)$ the set formed by the connected components of a topological space $X$, and let $\mathscr{A}$ be a subspace arrangement of $T$. The set $\displaystyle L_{\mathscr{A}}^c := L_{\mathscr{A}} \sqcup \bigg\{\mathrm{C}\Big(\bigcap_{H \in \mathscr{B}}H\Big)\ \bigg|\ \mathscr{B} \subseteq \mathscr{A},\, \bigcap_{H \in \mathscr{B}}H \neq \emptyset\bigg\}$ is for instance a meet-refinement of $L_{\mathscr{A}}$.

\begin{definition}
Let $\mathscr{A}$ be a subspace arrangement in a topological space $T$, and denote by $C_{\mathscr{A}}$ the set formed by the connected components of $\displaystyle T \setminus \bigcup_{H \in \mathscr{A}}H$. An element of $C_{\mathscr{A}}$ is called a \textbf{chamber} of $\mathscr{A}$.
\end{definition}

\noindent Consider a subspace arrangement $\mathscr{A}$, and a meet-refinement $L$ of $L_{\mathscr{A}}$. Let $\mathrm{D}(L)$ be the finite distributive lattice of sets generated by $L \sqcup C_{\mathscr{A}}$ through unions and intersections, that is $$\mathrm{D}(L) := \Big\{\bigcup_{A \in M}A \sqcup \bigcup_{X \in D}X\ \Big|\ M \subseteq L,\, D \subseteq C_{\mathscr{A}}\Big\}.$$
In that case, for $A,B \in \mathrm{D}(L)$, we have $A \vee B = A \cup B$ and $A \wedge B = A \cap B$.

\begin{lemma} \label{LejiD}
Let $\mathscr{A}$ be a subspace arrangement in a topological space $T$, and $L$ a meet-refinement of $L_{\mathscr{A}}$. Then, $\mathrm{ji}\big(\mathrm{D}(L)\big) \subseteq \{\emptyset\} \sqcup L \sqcup C_{\mathscr{A}}$.
\end{lemma}

\begin{proof}
Every element of $\mathrm{D}(L) \setminus (\{\emptyset\} \sqcup L \sqcup C_{\mathscr{A}})$ is the union of at least two elements of $L \sqcup C_{\mathscr{A}}$. Then, none of them can be join-irreducible.
\end{proof}

\begin{theorem} \label{ThFuT}
Let $\mathscr{A}$ be a subspace arrangement in a topological space $T$, $L$ a meet-refinement of $L_{\mathscr{A}}$, and $f$ a valuation on $\mathrm{D}(L)$. Then,
$$\sum_{C \in C_{\mathscr{A}}}f(C) = \sum_{X \in L \sqcup \{\emptyset\}} \mu_{L \sqcup \{\emptyset\}}(X,T)f(X).$$
\end{theorem}

\begin{proof}
Note first that $T \in L$ but $T \notin \mathrm{ji}\big(\mathrm{D}(L)\big)$ as $\displaystyle T = \bigcup_{H \in \mathscr{A}}H \sqcup \bigcup_{C \in C_{\mathscr{A}}}C$. From Corollary~\ref{CoUma} and Lemma~\ref{LejiD}, we get
$$\sum_{A \in \{\emptyset\} \sqcup L \sqcup C_{\mathscr{A}}} \mu_{\{\emptyset\} \sqcup L \sqcup C_{\mathscr{A}}}(A,T)f(A) = 0.$$
The result is finally obtained after taking into account the following remarks:
\begin{itemize}
\item if $C \in C_{\mathscr{A}}$, then $\mu_{\{\emptyset\} \sqcup L \sqcup C_{\mathscr{A}}}(C,T) = -\mu_{\{\emptyset\} \sqcup L \sqcup C_{\mathscr{A}}}(C,C) = -1$,
\item if $X \in \{\emptyset\} \sqcup L$, then $[X,T] \cap C_{\mathscr{A}} = \emptyset$, hence $\mu_{\{\emptyset\} \sqcup L \sqcup C_{\mathscr{A}}}(X,T) = \mu_{\{\emptyset\} \sqcup L}(X,T)$.
\end{itemize}
\end{proof}

\begin{definition}
Let $T$ be a topological space, and denote by $H_n(T)$ the $n^{\text{th}}$ singular homology group of $T$ for $n \in \mathbb{N}$. The \textbf{Euler characteristic} of $T$ is
$$\chi(T) := \sum_{n \in \mathbb{N}} (-1)^n\, \mathrm{rank}\,H_n(T).$$ 
\end{definition}

\noindent We can now state the fundamental theorem of dissection theory.

\begin{corollary}[Fundamental Theorem of Dissection Theory] Let $\mathscr{A}$ be a subspace arrangement in a topological space $T$ with $|\chi(T)| < \infty$, and $L$ a meet-refinement of $L_{\mathscr{A}}$. Then, $$\sum_{C \in C_{\mathscr{A}}}\chi(C) = \sum_{X \in L} \mu_L(X,T) \chi(X).$$
\end{corollary}

\begin{proof}
It is known that $\chi(A) + \chi(B) = \chi(A \cup B) + \chi(A \cap B)$, for $A,B \subseteq T$, like stated at the end of \cite[§~12.4]{To}. The Euler characteristic is then a valuation on $\mathrm{D}(L)$. Moreover, $\chi(\emptyset)=0$ by definition. We consequently obtain the result by using Theorem~\ref{ThFuT} with $\chi$.
\end{proof}

\begin{example} \label{ExSp}
Consider the arrangement $\mathscr{A}$ of parametric $1$-spheres $\displaystyle H_1: \begin{cases}
		x = \cos(\frac{\pi}{4})  \\
		y = \sin(\frac{\pi}{4}) \cos(t), \\
		z = \sin(\frac{\pi}{4}) \sin(t)
\end{cases}$ $\displaystyle H_2: \begin{cases}
		x = -\cos(\frac{\pi}{8}) \\
		y = \sin(\frac{\pi}{8}) \cos(t), \\
		z = \sin(\frac{\pi}{8}) \sin(t)
\end{cases}$ $\displaystyle H_3: \begin{cases}
		x = \cos(\frac{\pi}{6}) \sin(t) \\
		y = \cos(\frac{\pi}{6}) \cos(t), \\
		z = \sin(\frac{\pi}{6})
\end{cases}$ $\displaystyle H_4: \begin{cases}
		x = \cos(\frac{\pi}{3}) \sin(t) \\
		y = \cos(\frac{\pi}{3}) \cos(t), \\
		z = -\sin(\frac{\pi}{3})
\end{cases}$ where $t \in [0,2\pi]$, in $\mathbb{S}^2$ represented on Figure~\ref{Ex2}. On one side, $C_{\mathscr{A}}$ has $6$ chambers having Euler characteristic $1$, and $1$ with Euler characteristic $0$, then $\displaystyle \sum_{C \in C_{\mathscr{A}}}\chi(C) = 6$. On the other side,
\begin{align*}
\sum_{X \in L_{\mathscr{A}}} \mu_{L_{\mathscr{A}}}(X,\mathbb{S}^2) \chi(X) = \, & \mu_{L_{\mathscr{A}}}(\mathbb{S}^2,\mathbb{S}^2)\chi(\mathbb{S}^2) + \sum_{i \in [4]} \mu_{L_{\mathscr{A}}}(H_i,\mathbb{S}^2)\chi(H_i) \\
& + \mu_{L_{\mathscr{A}}}(H_1 \cap H_3,\mathbb{S}^2)\chi(H_1 \cap H_3) + \mu_{L_{\mathscr{A}}}(H_2 \cap H_3,\mathbb{S}^2)\chi(H_2 \cap H_3) \\
= \, & 1 \times 2 + 4 \times (-1) \times 0 + 1 \times 2 + 1 \times 2 \\
= \, & 6.
\end{align*}
	
\begin{figure}[h]
		\centering
		\includegraphics[scale=0.45]{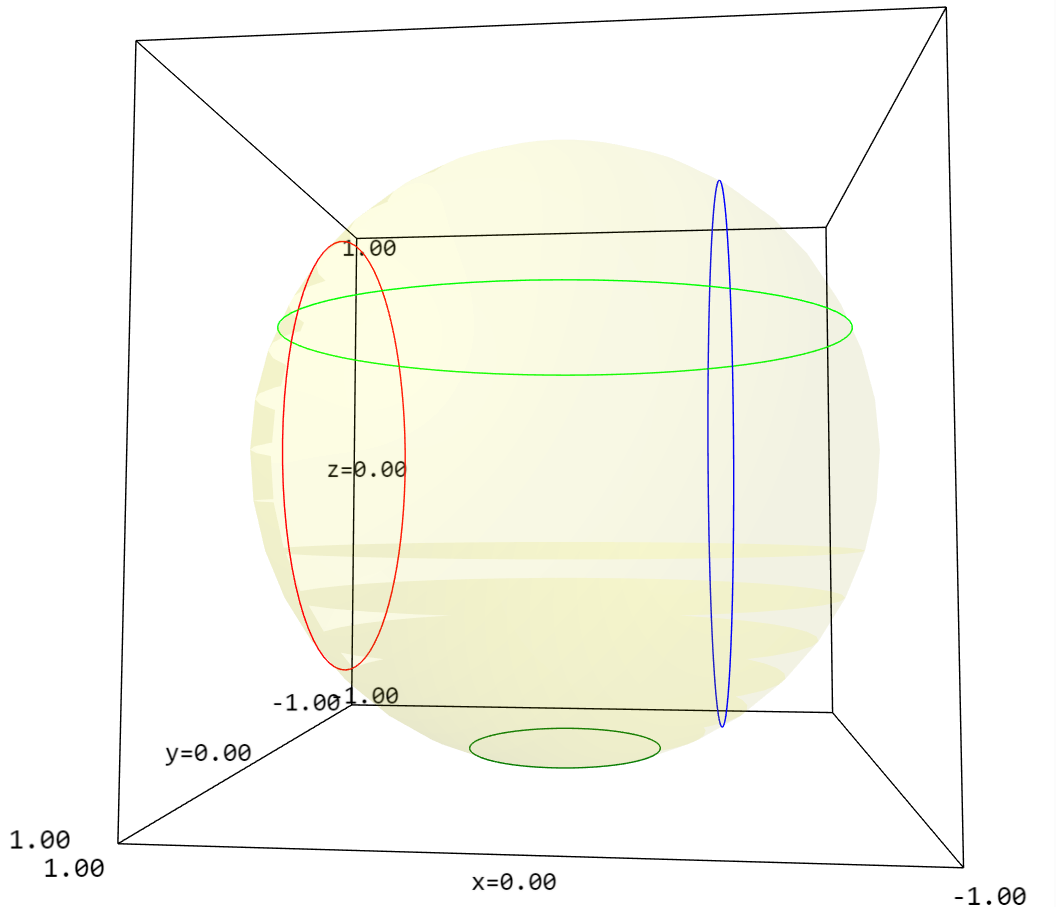}
		\caption{$1$-Sphere Arrangement of Example~\ref{ExSp}}
		\label{Ex2}
\end{figure}
\end{example}

\begin{corollary} \label{CoNCh}
Let $\mathscr{A}$ be a subspace arrangement in a topological space $T$ with $|\chi(T)| < \infty$, and $L$ a meet-refinement of $L_{\mathscr{A}}$. Suppose that every chamber of $\mathscr{A}$ has the same Euler characteristic $c \neq 0$. Then, $$\#C_{\mathscr{A}} = \frac{1}{c} \sum_{X \in L} \mu_L(X,T) \chi(X).$$
\end{corollary}

\begin{proof}
It is obviously a consequence of the fundamental theorem of dissection theory where $\chi(C) = c$ for $C \in C_{\mathscr{A}}$.
\end{proof}

\section{Face Counting for Submanifold Arrangement} \label{SecCo}

\noindent We use the fundamental theorem of dissection theory to compute the $\mathrm{f}$-polynomial of submanifold arrangements having specific face properties.

\begin{definition}
	Let $\mathscr{A}$ be a subspace arrangement in topological space $T$, and $X \in L_{\mathscr{A}}$. The \textbf{induced subspace arrangement} on $X$ is the subspace arrangement in $X$ defined by
	$$\mathscr{A}_X := \big\{H \cap X\ \big|\ H \in \mathscr{A},\, H \cap X \notin \{\emptyset, X\}\big\}.$$
	Let $\displaystyle F_{\mathscr{A}} := \bigsqcup_{X \in L_{\mathscr{A}}}C_{\mathscr{A}_X}$, and call an element of $F_{\mathscr{A}}$ a \textbf{face} of $\mathscr{A}$.
\end{definition}

\begin{definition}
Recall that a \textbf{$n$-dimensional manifold} or \textbf{$n$-manifold} is a topological space with the property that each point has a neighborhood that is homeomorphic to $\mathbb{R}^n$, and a \textbf{submanifold} of a $n$-manifold $T$ is a $k$-manifold included in $T$ where $k \in [0,n]$.
\end{definition}

\begin{definition}
Let us call \textbf{submanifold arrangement} in the $n$-manifold $T$ a finite set of submanifolds $\mathscr{A}$ in $T$ such that every element of $L_{\mathscr{A}} \cup F_{\mathscr{A}}$ is a submanifold of $T$.
\end{definition}

\begin{example} \label{ExR2}
Consider the arrangement $\mathscr{A}$ of $1$-manifolds $H_1: y = 6 \sin(x)$, $H_2: y = x + \cos(x)$, $\displaystyle H_3: \frac{x^2}{64} + \frac{y^2}{25} = 1$ in $\mathbb{R}^2$ represented on Figure~\ref{Ex1}. We see that
\begin{align*}
	\sum_{X \in L_{\mathscr{A}}} \mu_{L_{\mathscr{A}}}(X,\mathbb{R}^2) \chi(X) = \, & \mu_{L_{\mathscr{A}}}(\mathbb{R}^2,\mathbb{R}^2)\chi(\mathbb{R}^2) + \mu_{L_{\mathscr{A}}}(H_1,\mathbb{R}^2)\chi(H_1) + \mu_{L_{\mathscr{A}}}(H_2,\mathbb{R}^2)\chi(H_2) \\
	& + \mu_{L_{\mathscr{A}}}(H_3,\mathbb{R}^2)\chi(H_3) + \mu_{L_{\mathscr{A}}}(H_1 \cap H_2,\mathbb{R}^2)\chi(H_1 \cap H_2) \\
	& + \mu_{L_{\mathscr{A}}}(H_1 \cap H_3,\mathbb{R}^2)\chi(H_1 \cap H_3) + \mu_{L_{\mathscr{A}}}(H_2 \cap H_3,\mathbb{R}^2)\chi(H_2 \cap H_3) \\
	= \, & 1 \times 1 + (-1) \times (-1) + (-1) \times (-1) + (-1) \times 0 + 1 \times 3 + 1 \times 10 + 1 \times 2 \\
	= \, & 18
\end{align*}
is the number of chamber in $C_{\mathscr{A}}$.
	
\begin{figure}[h]
	\centering
	\includegraphics[scale=0.8]{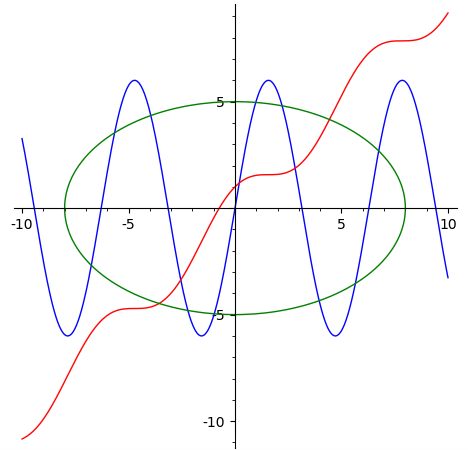}
	\caption{Submanifold Arrangement of Example~\ref{ExR2}}
	\label{Ex1}
\end{figure}
\end{example}

\begin{definition}
Let $\mathscr{A}$ be a submanifold arrangement in a $n$-manifold $T$, and $x$ a variable. For $k \in [0,n]$, denote by $\mathrm{f}_k(\mathscr{A})$ the number of $k$-dimensional faces of $\mathscr{A}$. The \textbf{$\mathrm{f}$-polynomial} of $\mathscr{A}$ is
$$\mathrm{f}_{\mathscr{A}}(x) := \sum_{k \in [0,n]} \mathrm{f}_k(\mathscr{A}) x^{n-k}.$$
\end{definition}

\begin{proposition} \label{PrSub}
Let $\mathscr{A}$ be a submanifold arrangement in a $n$-manifold $T$ with $|\chi(T)| < \infty$. Suppose that
\begin{align*}
& \forall k \in [0,n],\, \forall X \in L_{\mathscr{A}},\, \dim X = k:\ \chi(X) = l_k, \\	
& \forall k \in [0,n],\, \forall C \in F_{\mathscr{A}},\, \dim C = k:\ \chi(C) = c_k \neq 0.
\end{align*}
Then,
$$\mathrm{f}_{\mathscr{A}}(x) = \sum_{i \in [0,n]}  \sum_{\substack{Y \in L_{\mathscr{A}} \\ \dim Y = i}} \sum_{k \in [0,i]} \sum_{\substack{X \in L_{\mathscr{A}_Y} \\ \dim X = k}} \frac{l_k}{c_i} \mu_{L_{\mathscr{A}}}(X,Y) x^{n-k}.$$
\end{proposition}

\begin{proof}
Using the fundamental theorem of dissection theory, we get
\begin{align*}
\mathrm{f}_i(\mathscr{A}) & = \sum_{\substack{Y \in L_{\mathscr{A}} \\ \dim Y = i}} \#C_{L_{\mathscr{A}_Y}} \\
& = \frac{1}{c_i} \sum_{\substack{Y \in L_{\mathscr{A}} \\ \dim Y = i}} \sum_{X \in L_{\mathscr{A}_Y}} \mu_{L_{\mathscr{A}_Y}}(X,Y) \chi(X)	\\
& = \sum_{\substack{Y \in L_{\mathscr{A}} \\ \dim Y = i}} \sum_{k \in [0,i]} \sum_{\substack{X \in L_{\mathscr{A}_Y} \\ \dim X = k}} \frac{l_k}{c_i} \mu_{L_{\mathscr{A}_Y}}(X,Y) \\
& = \sum_{\substack{Y \in L_{\mathscr{A}} \\ \dim Y = i}} \sum_{k \in [0,i]} \sum_{\substack{X \in L_{\mathscr{A}_Y} \\ \dim X = k}} \frac{l_k}{c_i} \mu_{L_{\mathscr{A}}}(X,Y).
\end{align*}
\end{proof}

\begin{definition}
Let $\mathscr{A}$ be a submanifold arrangement in a $n$-manifold $T$. The \textbf{rank} of $X \in L_{\mathscr{A}}$ is $\mathrm{rk}\,X := n - \dim X$, and that of $\mathscr{A}$ is $\mathrm{rk}\,\mathscr{A} := \max \{\mathrm{rk}\,X\ |\ X \in L_{\mathscr{A}}\}$.
\end{definition}

\begin{definition}
Let $\mathscr{A}$ be a submanifold arrangement in a $n$-manifold $T$, and $x,y$ two variables. The \textbf{Möbius Polynomial} of $\mathscr{A}$ is
$$\mathrm{M}_{\mathscr{A}}(x,y) := \sum_{X,Y \in L_{\mathscr{A}}} \mu_{L_{\mathscr{A}}}(X,Y) x^{\mathrm{rk}\,X} y^{\mathrm{rk}\,\mathscr{A} - \mathrm{rk}\,Y}.$$
\end{definition}

\begin{corollary}
Let $\mathscr{A}$ be a submanifold arrangement in a $n$-manifold $T$ with $|\chi(T)| < \infty$. Suppose that $\chi(X) = (-1)^{\dim X}$ for every $X \in L_{\mathscr{A}} \cup F_{\mathscr{A}}$. Then,
$$\mathrm{f}_{\mathscr{A}}(x) = (-1)^{\mathrm{rk}\,\mathscr{A}} \mathrm{M}_{\mathscr{A}}(-x,-1).$$
\end{corollary}

\begin{proof}
From Proposition~\ref{PrSub}, we obtain
\begin{align*}
\mathrm{f}_{\mathscr{A}}(x) & = \sum_{i \in [0,n]}  \sum_{\substack{Y \in L_{\mathscr{A}} \\ \dim Y = i}} \sum_{k \in [0,i]} \sum_{\substack{X \in L_{\mathscr{A}_Y} \\ \dim X = k}} (-1)^{k-i} \mu_{L_{\mathscr{A}}}(X,Y) x^{n-k} \\
& = \sum_{Y \in L_{\mathscr{A}}} \sum_{X \in L_{\mathscr{A}_Y}} (-1)^{\dim X - \dim Y} \mu_{L_{\mathscr{A}}}(X,Y) x^{n- \dim X} \\
& = \sum_{Y \in L_{\mathscr{A}}} \sum_{X \in L_{\mathscr{A}_Y}} (-1)^{n - \dim Y} \mu_{L_{\mathscr{A}}}(X,Y) (-1)^{\dim X - n} x^{n- \dim X} \\
& = \sum_{Y \in L_{\mathscr{A}}} \sum_{X \in L_{\mathscr{A}_Y}} (-1)^{\mathrm{rk}\,Y} \mu_{L_{\mathscr{A}}}(X,Y) (-x)^{\mathrm{rk}\,X} \\
& = (-1)^{\mathrm{rk}\,\mathscr{A}} \sum_{Y \in L_{\mathscr{A}}} \sum_{X \in L_{\mathscr{A}_Y}}  \mu_{L_{\mathscr{A}}}(X,Y) (-x)^{\mathrm{rk}\,X} (-1)^{\mathrm{rk}\,Y - \mathrm{rk}\,\mathscr{A}} \\
& = (-1)^{\mathrm{rk}\,\mathscr{A}} \mathrm{M}_{\mathscr{A}}(-x,-1).
\end{align*}
\end{proof}

\begin{corollary}
Let $\mathscr{A}$ be a submanifold arrangement in a $n$-manifold $T$ with $|\chi(T)| < \infty$. Suppose that
$$\forall C \in F_{\mathscr{A}}:\, \chi(C) = (-1)^{\dim C} \quad \text{and} \quad \forall X \in L_{\mathscr{A}}:\, \chi(X) = \begin{cases}
	2 & \text{if}\ \dim X \equiv 0 \mod 2 \\
	0 & \text{otherwise}	
\end{cases}.$$
Moreover, define $\displaystyle \gamma_n := \begin{cases}
	1 & \text{if}\ \dim X \equiv 0 \mod 2 \\
   -1 & \text{otherwise}	
\end{cases}$. Then,
$$\mathrm{f}_{\mathscr{A}}(x) = (-1)^{n-\mathrm{rk}\,\mathscr{A}} \big(\mathrm{M}_{\mathscr{A}}(x,-1) + \gamma_n \mathrm{M}_{\mathscr{A}}(-x,-1)\big).$$
\end{corollary}

\begin{proof}
	From Proposition~\ref{PrSub}, we obtain
	\begin{align*}
		\mathrm{f}_{\mathscr{A}}(x) & = \sum_{i \in [0,n]}  \sum_{\substack{Y \in L_{\mathscr{A}} \\ \dim Y = i}} \sum_{k \in [0,i]} \sum_{\substack{X \in L_{\mathscr{A}_Y} \\ \dim X = k}} (-1)^{-i} \chi(X) \mu_{L_{\mathscr{A}}}(X,Y) x^{n-k} \\
		& = \sum_{Y \in L_{\mathscr{A}}} \sum_{X \in L_{\mathscr{A}_Y}} (-1)^{-\dim Y} \chi(X) \mu_{L_{\mathscr{A}}}(X,Y) x^{n- \dim X} \\
		& = (-1)^n \sum_{Y \in L_{\mathscr{A}}} \sum_{X \in L_{\mathscr{A}_Y}} \chi(X) \mu_{L_{\mathscr{A}}}(X,Y) x^{\mathrm{rk}\,X} (-1)^{\mathrm{rk}\,Y} \\
		& = (-1)^{n - \mathrm{rk}\,\mathscr{A}} \sum_{Y \in L_{\mathscr{A}}} \sum_{X \in L_{\mathscr{A}_Y}} \chi(X) \mu_{L_{\mathscr{A}}}(X,Y) x^{\mathrm{rk}\,X} (-1)^{\mathrm{rk}\,\mathscr{A} -\mathrm{rk}\,Y} \\
		& = (-1)^{n - \mathrm{rk}\,\mathscr{A}} \sum_{Y \in L_{\mathscr{A}}} \sum_{X \in L_{\mathscr{A}_Y}} \mu_{L_{\mathscr{A}}}(X,Y) x^{\mathrm{rk}\,X} (-1)^{\mathrm{rk}\,\mathscr{A} -\mathrm{rk}\,Y} \\
		& \quad + (-1)^{n - \mathrm{rk}\,\mathscr{A}} \gamma_n \sum_{Y \in L_{\mathscr{A}}} \sum_{X \in L_{\mathscr{A}_Y}} \mu_{L_{\mathscr{A}}}(X,Y) (-x)^{\mathrm{rk}\,X} (-1)^{\mathrm{rk}\,\mathscr{A} -\mathrm{rk}\,Y} \\
		& = (-1)^{n-\mathrm{rk}\,\mathscr{A}} \mathrm{M}_{\mathscr{A}}(x,-1) + (-1)^{n-\mathrm{rk}\,\mathscr{A}} \gamma_n \mathrm{M}_{\mathscr{A}}(-x,-1).
	\end{align*}
\end{proof}

\bibliographystyle{abbrvnat}

\end{document}